\documentclass[twocolumn,final]{svjour3}

\usepackage{amsthm}
\usepackage{amsfonts}
\usepackage{mathtools}
\usepackage{dblfloatfix}
\usepackage{stanli}
\usepackage{times}
\usepackage{multirow}
\usepackage{subcaption}
\usepackage{setspace}
\usepackage[round]{natbib}
\usepackage{bm}
\usetikzlibrary{arrows.meta,plotmarks}
\usepackage[colorlinks=false]{hyperref}

\title{Global optimality in minimum compliance topology optimization of frames and shells by moment-sum-of-squares hierarchy}
\author{Marek Tyburec \and Jan Zeman \and Martin Kru\v{z}\'ik \and Didier Henrion}
\institute{Marek Tyburec\textsuperscript{1}\\
	\email{marek.tyburec@fsv.cvut.cz}\\
	ORCID: 0000-0003-0798-0948\\
	\\
	Jan Zeman\textsuperscript{1}\\
	\email{Jan.Zeman@cvut.cz}\\
	ORCID: 0000-0003-2503-8120\\
	\\
	Martin Kru\v{z}\'{i}k\textsuperscript{2}\\
	\email{kruzik@utia.cas.cz}\\
	ORCID: 0000-0003-1558-5809\\
	\\
	Didier Henrion\textsuperscript{3,4}\\
	\email{henrion@laas.fr}\\
	ORCID: 0000-0001-6735-7715\\
	\\\
	\textsuperscript{1}Czech Technical University in Prague, Faculty of Civil Engineering, Department of Mechanics, Th\'{a}kurova 7, 16629 Prague 6, Czech Republic\\
	\\
	\textsuperscript{2}Czech Technical University in Prague, Faculty of Civil Engineering, Department of Physics, Th\'{a}kurova 7, 16629 Prague 6, Czech Republic\\
	\\
	\textsuperscript{3}Czech Technical University in Prague, Faculty of Electrical Engineering, Department of Control Engineering, Karlovo n\'{a}m\v{e}st\'{i} 13, 12135 Prague 2, Czech Republic\\
	\\
	\textsuperscript{4}LAAS-CNRS, 7 avenue du Colonel Roche, 31400 Toulouse, France
}

\newcommand*\circled[1]{\tikz[baseline=(char.base)]{
		\node[shape=circle,draw,inner sep=0pt,fill=white, minimum size=4mm] (char) {#1};}}

\newcommand{\hiddenbox}[1]{}
\newtheorem{assumption}{Assumption}

\begin{document}
\tolerance=2500
\maketitle

\begin{abstract}
The design of minimum-compliance bending-resistant structures with continuous cross-section parameters is a~challenging task because of its inherent non-convexity. Our contribution develops a strategy that facilitates computing all guaranteed globally optimal solutions for frame and shell structures under multiple load cases and self-weight. To this purpose, we exploit the fact that the stiffness matrix is usually a polynomial function of design variables, allowing us to build an equivalent non-linear semidefinite programming formulation over a semi-algebraic feasible set. This formulation is subsequently solved using the Lasserre moment-sum-of-squares hierarchy, generating a~sequence of outer convex approximations that monotonically converges from below to the optimum of the original problem. Globally optimal solutions can subsequently be extracted using the Curto-Fialkow flat extension theorem. Furthermore, we show that a simple correction to the solutions of the relaxed problems establishes a feasible upper bound, thereby deriving a simple sufficient condition of global $\varepsilon$-optimality. When the original problem possesses a~unique minimum, we show that this solution is found with a~zero optimality gap in the limit. These theoretical findings are illustrated on several examples of topology optimization of frames and shells, for which we observe that the hierarchy converges in a finite (rather small) number of steps.
\end{abstract}

\keywords{discrete topology optimization, frame structures, shell structures, semidefinite programming, polynomial optimization, global optimality}

\section{Introduction}

Structural optimization is a research field developing concepts for the design of efficient structures. It was pioneered by \citet{Michell_1904}, who showed that minimum-weight truss structures under a~single load case are fully-stressed, and the optimal trajectories of their bars align with the principal stress directions. Hence, optimal designs can contain an infinite number of bars in general. This drawback, which hinders their manufacturability, was overcome in the work of \citet{Dorn1964} by introducing the ground structure approach, effectively discretizing the continuum into a~finite set of potential nodes and their interconnections of finite elements. Because the dimensionality of these potential elements is lower than that of the continuum and the presence and sizing of each element are investigated, this setting is referred to as discrete topology optimization.

\begin{figure*}[!t]
	\begin{subfigure}{0.25\linewidth}
		\begin{tikzpicture}
		\scaling{2}
		\point{a}{0.000000}{0.000000}
		\notation{1}{a}{\circled{$1$}}[below right=0mm]
		\point{b}{1.000000}{0.500000}
		\notation{1}{b}{\circled{$2$}}[above right=0mm]
		\point{c}{0.000000}{1.000000}
		\notation{1}{c}{\circled{$3$}}[above right=0mm]
		\beam{2}{a}{b}
		\notation{4}{a}{b}[$1$]
		\beam{2}{b}{c}
		\notation{4}{b}{c}[$2$]
		\support{3}{a}[270]
		\support{3}{c}[270]
		\point{d1}{0.000000}{-0.250000}
		\point{d2}{1.000000}{-0.250000}
		\dimensioning{1}{d1}{d2}{-1.000000}[$1.0$]
		\point{e1}{1.250000}{0.000000}
		\point{e2}{1.250000}{0.500000}
		\dimensioning{2}{e1}{e2}{-0.75}[$0.5$]
		\point{e3}{1.250000}{1.000000}
		\dimensioning{2}{e2}{e3}{-0.75}[$0.5$]
		\load{1}{b}[0][1.0][0.0]
		\notation{1}{b}{$1.6$}[right=9mm]
		\load{1}{b}[90][-0.625][0.0]
		\notation{1}{b}{$1.0$}[below=5.5mm]
		\end{tikzpicture}
		\caption{}
	\end{subfigure}%
	\hfill\begin{subfigure}{0.22\linewidth}
		\includegraphics[width=\linewidth]{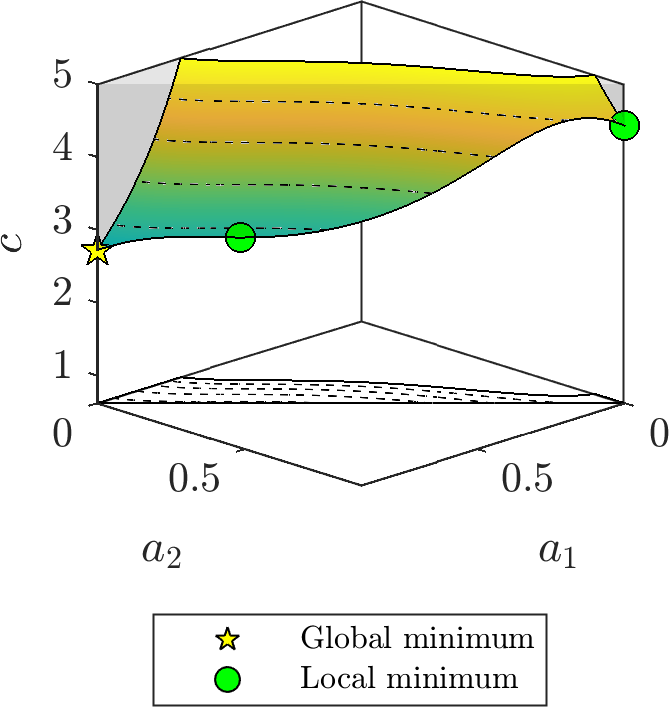}
		\caption{}
	\end{subfigure}%
	\hfill\begin{subfigure}{0.22\linewidth}
		\includegraphics[width=\linewidth]{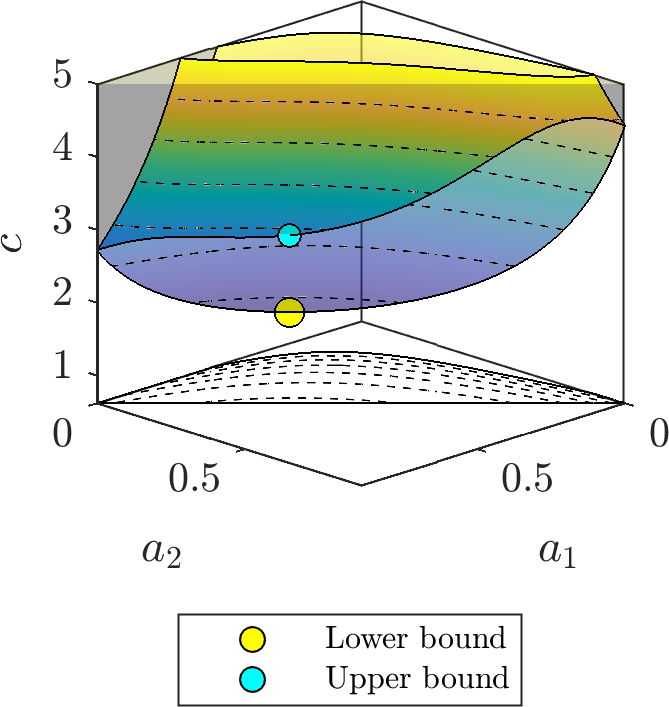}
		\caption{}
	\end{subfigure}%
	\hfill\begin{subfigure}{0.22\linewidth}
		\includegraphics[width=\linewidth]{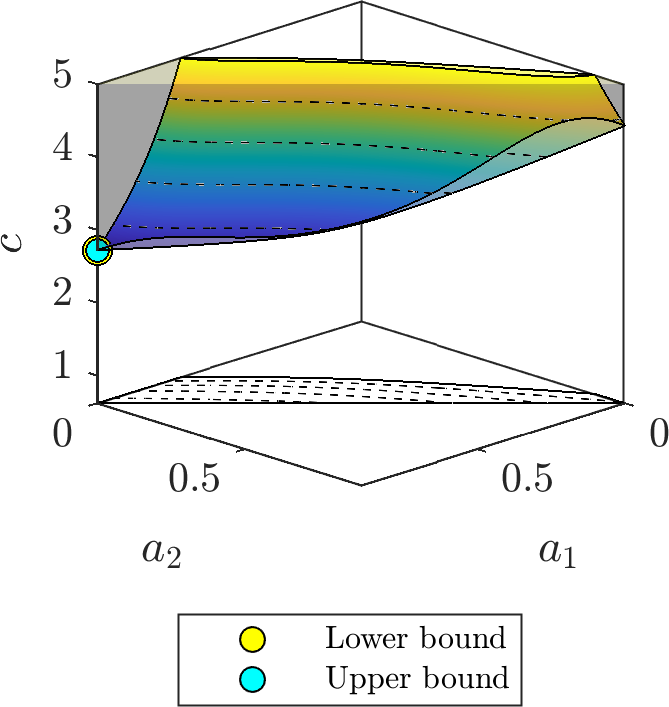}
		\caption{}
	\end{subfigure}
	\caption{(a) Boundary conditions for the motivating problem, (b) the sublevel set $c \le 5$ of its feasible design space, and its (c) first and (d) second convex outer approximations constructed by the moment-sum-of-squares approach. Variables $a_1$ and $a_2$ stand for the cross-section areas of the two elements and $c$ denotes the corresponding compliance (assuming moments of inertia $I_i = a_i^2$, $i \in \{1,2\}$).}
	\label{fig:motivation}
\end{figure*}

A tremendous progress has been made for the case of trusses. While \citet{Dorn1964} developed a linear programming formulation for the single-load-case plastic design, \citet{BENDSOE_1991} and \citet{Achtziger_1992} introduced a convex displacement-based elastic-design quadratic program that additionally allowed for multiple load cases. Its dual, which incorporates the cross-section variables explicitly, was shown by \citet{Lobo_1998} and \citet{Ben_Tal_2001} to be a second-order conic program. This latter formulation handles multiple load cases with stress constraints efficiently \citep{Tyburec_2020}. Convexity prevails even for fundamental free-vibration constraints, in which case semidefinite programming can be used \citep{Achtziger_2008, Tyburec_2019}.

A completely different situation holds for bending-resistant structures with continuous design variables. To the best of our knowledge, no convex formulation has been established so far, and, therefore, solely local optimization techniques have been used. Among these, \citet{Saka_1980} developed a sequential linear programming approach to design minimum-weight framed structures, and \citet{Wang_2006} improved over its solution efficiency by using sequential quadratic programming instead. Another, relaxation-based sequential semidefinite programming method was proposed by \citet{Yamada_2015} to deal with vibration problems. Nonlinear programming \citep{Fredricson_2003}, Optimality Criteria (OC) \citep{Khan_1984,Chan_1995}, the Method of Moving Assymptotes (MMA) \citep{Svanberg1987,Fredricson_2005}, and meta-heuristics \citep{An_2017} are other commonly used alternatives.

Except for our earlier conference paper \citep{Tyburec_2020b}, which is a very preliminary version of this manu\-script, the only global approach that can be found in the literature considers a discrete setting of the problem, in which case the cross-sections are selected from a predefined catalog, allowing to use the branch-and-bound method \citep{Kanno_2016}.

\subsection{Motivation}\label{sec:motivation}

This contribution investigates a~conceptual design of bending-resistant structures with continuous design variables. As follows from the previous survey, only local optimization approaches have been adopted so far to tackle the frame/shell structure optimization problem. It is, therefore, not surprising that they fail to converge to globally optimal solutions even for toy problems, such as the one shown in Fig.~\ref{fig:motivation}a.

Consider a linear-elastic material with the dimensionless Young modulus $E=1$, available material volume $\overline{V}=1$, element cross-sections parameterized by their area $a_i$ and the moment of inertia $I_i(a_i) = a_i^2$, which corresponds to rectangular cross-sections with the height-to-width ratio of $12$, for example. Our goal is to find the cross-section areas $a_i$ associated with the elements $i = \{1,2\}$ that induce minimum compliance $c$ within all non-negative $a_i$ satisfying the volume bound $\overline{V}$. Standard local optimization techniques such as OC, MMA, and \textsc{Matlab} inbuilt optimizer \texttt{fmincon} all converge\footnote{For OC and MMA, we adopted the commonly-used starting point of uniform mass distribution, i.e., $a_1=a_2=0.2\sqrt{5}$. For \texttt{fmincon}, the default starting point was used.} to the optimized compliance $c = 2.895$ and the corresponding areas $a_1 = 0.652$ and $a_2 = 0.242$. However, the globally optimal design possesses compliance $c^* = 2.719$, which requires $a_1^* = 0.4 \sqrt{5}$ and $a_2^* = 0$. 

This problem exhibits three local optima in particular, with the last one being $a_1 = 0$, $a_2 = 0.4\sqrt{5}$ and $c=4.429$, see Fig.~\ref{fig:motivation}b. Thus, the global optimum may be reached by examining a~few different starting points in the optimizers. Such a~procedure, however, cannot assure global optimality and cannot neither assess quality of the optimized designs with respect to the global optimum. When the number of structural elements increases, design-dependent loads are considered, and higher-order polynomials for the moments of inertia are used, finding globally-optimal minimum-compliance designs becomes extremely challenging.

\subsection{Aims and novelty}

We address this problem by exploiting the simple fact that the constraints can be formulated as polynomial functions, hence forming a~(basic) semi-algebraic feasible set. Using a polynomial objective function in addition, the moment-sum-of-squares (Lasserre) hierarchy of convex outer approximations (relaxations) can be used to solve and extract all globally-optimal solutions. These relaxations provide a~non-decreasing sequence of lower bounds, eventually attaining the optimal compliance, Figs \ref{fig:motivation}c-\ref{fig:motivation}d. In addition, we show how to correct such obtained lower-bound designs, and hence generate feasible upper bounds. A comparison of these bounds then assesses the design quality, and their equality establishes a simple sufficient condition of global optimality. We further show that when a unique global optimum exists, we can expect the occurrence of a~bound equality. Fortunately, this situation occurs quite often when the design domain lacks structural and boundary conditions symmetries.

This paper is organized as follows. In Section \ref{sec:mom-sos}, we introduce polynomial optimization and the moment-sum-of-squares hierarchy. Section \ref{sec:nsdp} develops a~non-linear semidefinite programming formulation for topology optimization of frame structures, which we modify subsequently for the moment-sum-of-squares hierarchy in Section \ref{sec:eff}. Section \ref{sec:ub} reveals how to correct the lower-bound designs generated by the hierarchy to obtain feasible upper bounds, and Section \ref{sec:opt} introduces the sufficient condition of global $\varepsilon$-optimality as well as a~zero optimality gap for the case of a unique global optimum. Section~\ref{sec:shells} outlines the required changes in notation to allow for thickness optimization of shell structures. These results are then illustrated on five selected optimization problems in Section~\ref{sec:examples}. We finally summarize our contributions in Section~\ref{sec:conclusion}.

\section{Moment-sum-of-squares hierarchy}\label{sec:mom-sos}

In this section, we briefly outline the moment-sum-of-squares hierarchy when applied to solution of problems with polynomial matrix inequalities. For more information, we refer the reader to an expository text \citep{Henrion_2006} and to the excellent books \citep{Anjos2012,Lasserre_2015}.

Suppose we aim to solve an optimization problem of the form
\begin{subequations}\label{eq:po}
	\begin{alignat}{2}
	f^* = \inf_{\mathbf{x}}\, & f(\mathbf{x})\\
	\mathrm{s.t.}\; & \mathbf{G} (\mathbf{x}) &{}\succeq{} 0,\label{eq:po_pmi}
	\end{alignat}
\end{subequations}
where $f (\mathbf{x}): \mathbb{R}^n \rightarrow \mathbb{R}$ is a real polynomial function and $\mathbf{G} (\mathbf{x}): \mathbb{R}^n \rightarrow \mathbb{S}^{m}$ is a~real polynomial mapping, so that $\forall i, j: G_{i,j}(\mathbf{x}) = G_{j,i}(\mathbf{x})$ are real polynomial functions of $\mathbf{x}$. The degree of these polynomials is less than or equal to $k \in \mathbb{N}$. The symbol $\mathbb{S}^{m}$ denotes the space of real symmetric square matrices of size $m$, and ``$\succeq$'' establishes an ordering of fundamental eigenvalues, i.e., $\mathbf{G}(\mathbf{x})$ in \eqref{eq:po_pmi} is positive semidefinite. Hence, we call \eqref{eq:po_pmi} a polynomial matrix inequality (PMI) in what follows and denote its feasible set by $\mathcal{K}(\mathbf{G})$.

Clearly, the nonlinear semidefinite program \eqref{eq:po} covers a~variety of convex optimization problems as special cases, including linear and quadratic programming or linear semidefinite programming, see, e.g., \citep[Section 4.2]{Ben_Tal_2001}. Although these instances can be solved in a polynomial time, hence efficiently, \eqref{eq:po} exhibits $\mathcal{NP}$-hardness in general. This can be seen, for example, by a~reduction from binary programming, in which case the main diagonal of $\mathbf{G}(\mathbf{x})$ contains both $x_i^2-x_i$ and $x_i - x_i^2$ terms for all $i \in \{1\dots n\}$.

Despite the fact that the admissible set of $\mathbf{x}$ is generally non-convex, \eqref{eq:po}~admits an equivalent reformulation to a~convex optimization problem over a finite-dimensional cone of polynomials $C_k(\mathcal{K}(\mathbf{G}))$ of degree at most $k$ which are non-negative on $\mathcal{K}(\mathbf{G})$, i.e.,
\begin{equation}\label{eq:non-neg_cone}
\begin{aligned}
f^* &= \sup_{\lambda} \{\lambda: f(\mathbf{x})-\lambda \ge 0, \forall \mathbf{x} \in \mathcal{K}(\mathbf{G})\}\\ 
&= \sup_{\lambda}\, (f-\lambda),  \mathrm{s.t.\;} (f-\lambda) \in C_k(\mathcal{K}(\mathbf{G})).
\end{aligned}
\end{equation}
Unfortunately, it is not known how to handle $C_k(\mathcal{K}(\mathbf{G}))$ simply and in a tractable way.

To introduce an approach that allows to solve \eqref{eq:non-neg_cone}, we first adopt the following notation. Let $\mathbf{x} \mapsto \mathbf{b}_k (\mathbf{x})$ be the polynomial space basis of polynomials of degree at most $k$,
\begin{equation}
\begin{multlined}
\mathbf{b}_k(\mathbf{x}) = \left(1\;\; x_1\;\; x_2\;\; \dots\;\; x_n\;\; x_1^2\;\; x_1 x_2\;\; \dots\;\; x_1 x_n\right.\\ \left. x_2^2\;\; x_2 x_3\;\; \dots\;\; x_n^2\;\; \dots\;\; x_2^3\;\; \dots\;\; x_n^k
\right),
\end{multlined}
\end{equation}
Then, any polynomial $p(\mathbf{x})$ of degree at most $k$ can be written as
\begin{equation}
p (\mathbf{x}) = \mathbf{q}^\mathrm{T} \mathbf{b}_k (\mathbf{x}),
\end{equation}
in which $\mathbf{q}$ denotes a vector of coefficients associated with the basis $\mathbf{b}_k(\mathbf{x})$.

\begin{definition}\label{def:sos}
	The polynomial matrix $\bm{\Sigma}(\mathbf{x}): \mathbb{R}^n \rightarrow \mathbb{S}^m$ is a~(matrix) sum-of-squares (SOS) if there exists a polynomial matrix $\mathbf{H}(\mathbf{x}): \mathbb{R}^n \rightarrow \mathbb{R}^{m \times o}$, $o \in \mathbb{N}$, such that
	\begin{equation}
	\bm{\Sigma} (\mathbf{x})= \mathbf{H}(\mathbf{x})\left[\mathbf{H}(\mathbf{x})\right]^\mathrm{T}, \;\forall \mathbf{x} \in \mathbb{R}^n.
	\end{equation}
\end{definition}

Let $\langle \cdot, \cdot \rangle$ denote the standard inner product on matrices, $\bm{\alpha} \in \mathbb{N}^{\lvert \mathbf{b}_k(\mathbf{x}) \rvert}$ with $\mathbf{1}^\mathrm{T}\bm{\alpha} \le k$ be the multi-index associated with the basis $\mathbf{b}_k (\mathbf{x})$, and let $\mathbf{y} \in \mathbb{R}^{\lvert \mathbf{b}_k(\mathbf{x}) \rvert }$ be the moments (of probability measures supported on $\mathcal{K}(\mathbf{G}(\mathbf{x}))$) indexed in $\mathbf{b}_k (\mathbf{x})$. In what follows, we adopt the following notation for the elements of $\mathbf{y}$:
\begin{equation}
y_{\bm{\alpha}} = y_{\prod_{i=1}^n x_i^{\alpha_i}} \text{ is associated with } \prod_{i=1}^{n}x_i^{\alpha_i}.
\end{equation}
For example, when $\bm{\alpha} = (0\;\;0\;\;1\;\;2)^\mathrm{T}$, $y_{0012} = y_{x_3^1 x_4^2}$ corresponds to the polynomial $x_3^1 x_4^2 \in \mathbf{b}_k(\mathbf{x})$, where $k \ge 3$.
\begin{assumption}\label{ass:comp}\citep{Henrion_2006}
	Assume that there exist SOS polynomials $\mathbf{x} \mapsto p_0(\mathbf{x)}$ and $\mathbf{x} \mapsto \mathbf{R}(\mathbf{x})$ such that the superlevel set $\{\mathbf{x} \in \mathbb{R}^n : p_0(\mathbf{x}) + \langle \mathbf{R}(\mathbf{x}),\mathbf{G}(\mathbf{x})\rangle \geq 0\}$ is compact.
\end{assumption}
Note that Assumption \ref{ass:comp} is an algebraic certificate of compactness of the feasible set in problem \eqref{eq:po}. When Assumption \ref{ass:comp} holds, then, the dual of \eqref{eq:non-neg_cone} can be written equivalently as an infinite-dimensional generalized problem of moments, which is equipped with a finite-dimensional truncation:
\begin{subequations}\label{eq:truncmoment}
	\begin{alignat}{2}
	f^{(r)} = \min_{\mathbf{y}}\, && \mathbf{q}^\mathrm{T} \mathbf{y}\qquad\qquad\;\;\\
	\mathrm{s.t.}\;&& y_0 &= 1,\\
	&& \mathbf{M}_k (\mathbf{y}) &\succeq 0,\\
	&& \mathbf{M}_{k-d} (\mathbf{G}(\mathbf{x}) \mathbf{y}) &\succeq 0,
	\end{alignat}
\end{subequations}
in which $2r \ge k$, $r$ is the relaxation degree, and $d$ stands for the maximum degree of polynomials in $\mathbf{G}(\mathbf{x})$. In addition, $\mathbf{M}_k(\mathbf{y})$ and $\mathbf{M}_{k-d} (\mathbf{G}(\mathbf{x})\mathbf{y})$ are the truncated moment and localizing matrices associated with $\mathbf{y}$ and $\mathbf{G}(\mathbf{x})$. For a precise definition of $\mathbf{M}_k(\mathbf{y})$ and $\mathbf{M}_{k-d} (\mathbf{G}(\mathbf{x})\mathbf{y})$, we refer the reader to \citep{Henrion_2006}. These moment matrices are linear in $\mathbf{y}$, hence \eqref{eq:truncmoment} is a linear semidefinite program. Because \eqref{eq:truncmoment} is a~finite-dimensional convex relaxation of \eqref{eq:po}, we have $f^{(r)} \le f^*$, $\forall r \in \mathbb{N}$. Moreover, these relaxations are tighter with increasing $r$, making the sequence $\left(f^{(r)}\right)_{r}^{\infty}$ monotonically increasing and converging towards $f^*$.

\begin{theorem}\citep[Theorem 2.2]{Henrion_2006}\label{th:convergence}
	Let Assumption~\ref{ass:comp} be satisfied. Then, $f^{(r)} \uparrow f^*$ as $r \rightarrow \infty$ in \eqref{eq:truncmoment}.
\end{theorem}

Moreover, all globally optimal solutions of \eqref{eq:po} can be extracted from \eqref{eq:truncmoment} based on the flat extension theorem of \citet{Curto_1996}. Indeed, finite convergence occurs when
\begin{equation}\label{eq:rank}
s = \mathrm{Rank}(\mathbf{M}_k(\mathbf{y}^*)) = \mathrm{Rank}(\mathbf{M}_{k-d}(\mathbf{y}^*)),
\end{equation}
where $\mathbf{y}^*$ denotes the vector of optimal moments, and $s$ stands for the minimum number of distinct global minimizers \citep[Theorem 2.4]{Henrion_2006}.

\paragraph{Example} We illustrate the process of building the (Lasserre) moment-sum-of-squares hierarchy on an elementary example.
\begin{subequations}
\begin{alignat}{2}
\min_{a,c}\; && c\qquad\;\;\,\\
\mathrm{s.t.}\; && \begin{pmatrix}
c & \overline{f} \\
\overline{f} & a^2
\end{pmatrix} &\succeq 0,\\
&& \overline{V} - a &\ge 0,\\
&& a &\ge 0.
\end{alignat}
\end{subequations}
In the first relaxation, $\mathbf{y} = \begin{pmatrix} y_{00} & y_{10} & y_{01} & y_{20} & y_{11} & y_{02}\end{pmatrix}^\mathrm{T}$ is indexed in the polynomial space basis $\mathbf{b}_1 (a,c) = \begin{pmatrix}1 & c & a & c^2 & c a & a^2\end{pmatrix}^\mathrm{T}$. Then, the associated relaxation reads
\begin{subequations}
	\begin{alignat}{2}
	\min_{\mathbf{y}} && y_{10}\qquad\qquad\;\,\\
	\mathrm{s.t.} && \begin{pmatrix}
	y_{10} & \overline{f} \\
	\overline{f} & y_{02}
	\end{pmatrix} &\succeq 0,\\
	&& \overline{V} - y_{01} &\ge 0,\\
	&& y_{01} &\ge 0,\\
	&& y_{00} &= 1,\\
	&& \begin{pmatrix}
	y_{00} & y_{10} & y_{01}\\
	y_{10} & y_{20} & y_{11}\\
	y_{01} & y_{11} & y_{02}
	\end{pmatrix} &\succeq 0.
	\end{alignat}
\end{subequations}
For $r=2$, we have $\mathbf{y} = (y_{00}\;\,\allowbreak y_{10}\;\,\allowbreak y_{01}\;\,\allowbreak y_{20}\;\,\allowbreak y_{11}\;\,\allowbreak y_{02}\;\,\allowbreak y_{30}\;\,\allowbreak y_{21}\allowbreak y_{12}\;\,\allowbreak y_{03}\;\,\allowbreak y_{40}\;\,\allowbreak y_{31}\;\,\allowbreak y_{22}\;\,\allowbreak y_{13}\;\,\allowbreak y_{04})^\mathrm{T}$ indexed in $\mathbf{b}_2(a,c) = (1\allowbreak c\;\,\allowbreak a\;\,\allowbreak c^2\;\,\allowbreak c a\;\,\allowbreak a^2\;\,\allowbreak c^3\;\,\allowbreak c^2 a\;\,\allowbreak c a^2\;\,\allowbreak a^3\;\,\allowbreak c^4\;\,\allowbreak c^3 a\;\,\allowbreak c^2 a^2\;\,\allowbreak c a^3\;\,\allowbreak a^4)^\mathrm{T}$. The corresponding relaxation  is written as
\begin{subequations}
	\begin{alignat}{2}
	\min_{\mathbf{y}}\; && y_{10}\qquad\qquad\qquad\qquad\qquad\qquad\qquad\;\;\\
	\mathrm{s.t.}\; && \begin{pmatrix}
	y_{10} & \overline{f} & y_{20} & \overline{f}y_{10} & y_{11} & \overline{f} y_{01} \\
	\overline{f} & y_{02} & \overline{f}y_{10} & y_{12} & \overline{f} y_{01} & y_{03}\\
	y_{20} & \overline{f}y_{10} & y_{30} & \overline{f}y_{20} & y_{21} & \overline{f} y_{11} \\
	\overline{f}y_{10} & y_{12} & \overline{f}y_{20} & y_{22} & \overline{f} y_{11} & y_{13}\\
	y_{11} & \overline{f}y_{01} & y_{21} & \overline{f}y_{11} & y_{12} & \overline{f} y_{02} \\
	\overline{f}y_{01} & y_{03} & \overline{f}y_{11} & y_{13} & \overline{f} y_{02} & y_{04}
	\end{pmatrix} &\succeq 0,\hspace{-4mm}\\
	&& \begin{pmatrix}
	\overline{V} - y_{01} & \overline{V}y_{10} - y_{11} & \overline{V}y_{01} - y_{02}\\
	\overline{V}y_{10} - y_{11} & \overline{V}y_{20} - y_{21} & \overline{V}y_{11} - y_{12}\\
	\overline{V}y_{01} - y_{02} & \overline{V}y_{11} - y_{12} & \overline{V}y_{02} - y_{03}
	\end{pmatrix} &\succeq 0,\hspace{-4mm}\\
	&& \begin{pmatrix}
	y_{01} & y_{11} & y_{02}\\
	y_{11} & y_{21} & y_{12}\\
	y_{02} & y_{12} & y_{03}
	\end{pmatrix} &\succeq 0,\hspace{-4mm}\\
	&& y_{00} &= 1,\hspace{-4mm}\\
	&& \begin{pmatrix}
	y_{00} & y_{10} & y_{01} & y_{20} & y_{11} & y_{02}\\
	y_{10} & y_{20} & y_{11} & y_{30} & y_{21} & y_{12}\\
	y_{01} & y_{11} & y_{02} & y_{21} & y_{12} & y_{03}\\
	y_{20} & y_{30} & y_{21} & y_{40} & y_{31} & y_{22}\\
	y_{11} & y_{21} & y_{12} & y_{31} & y_{22} & y_{13}\\
	y_{02} & y_{12} & y_{03} & y_{22} & y_{13} & y_{04}
	\end{pmatrix} &\succeq 0.\hspace{-4mm}
	\end{alignat}
\end{subequations}
When solved, this relaxation allows for extracting the global solution of $a^*=\overline{V}$ and $c^*=\overline{f}^2/\overline{V}^2$.

\section{Methodology}

Topology optimization of discrete structures provides a natural application for the ground structure approach \citep{Dorn1964}, a~discretized design domain composed of a~fixed set of $n_\mathrm{n} \in \mathbb{N}$ nodes and their subsets of admissible $n_\mathrm{e} \in \mathbb{N}$ finite elements. Here, we employ the simplest two-node Euler-Bernoulli frame elements that adopt linear shape functions to interpolate the longitudinal displacements and cubic shape functions to interpolate the lateral displacements and rotations. Another elements can be adopted though, see Section \ref{sec:examples_cantilever} for applications to the Timoshenko beam element and the MITC4 shell element.

Each of these finite elements (indexed with $i$) must be supplied with the non-negative cross-section area $a_i \in \mathbb{R}_{\ge 0}$ and the area moment of inertia $I_i \in \mathbb{R}_{\ge 0}$. These are to be found in the optimization process. In this contribution, we assume, for convenience, that the moment of inertia is a second- or third-order polynomial function of the cross-sections, 
\begin{equation}\label{eq:inertia}
I_i(a_i) = c_\mathrm{II} a_i^2 + c_\mathrm{III} a_i^3,
\end{equation}
with $c_\mathrm{II}, c_\mathrm{III} \in \mathbb{R}_{\ge 0}$ being fixed constants. When $I_i (a_i) = 0$ and $a_i = 0$, the finite element vanishes and does not contribute to the load transfer.

Different topology optimization formulations exist, accommodating specific needs of particular applications. Here, we consider the problem of searching the minimum-compliant design under multiple load cases \eqref{eq:original_compliance} while satisfying the linear-elastic equilibrium equation \eqref{eq:original_equilibrium} and limiting the material volume from above by $\overline{V} \in \mathbb{R}_{>0}$ \eqref{eq:original_volume}. Physical admissibility of the resulting designs is ensured by the non-negative cross-section areas \eqref{eq:original_areas}. Combination of these ingredients establishes the basic elastic-design formulation
\begin{subequations}\label{eq:original}
	\begin{alignat}{2}
	&\min_{\mathbf{a}, \mathbf{u}_1,\dots,\mathbf{u}_{n_\mathrm{lc}}}\, & \sum_{j=1}^{n_\mathrm{lc}} \omega_j \mathbf{f}_j(\mathbf{a})^\mathrm{T} \mathbf{u}_j\;\;\label{eq:original_compliance}\\
	&\quad\;\;\mathrm{s.t.} & \mathbf{K}_j(\mathbf{a}) \mathbf{u}_j - \mathbf{f}_j(\mathbf{a}) &{}={}\mathbf{0},\;\forall j \in \{1\dots n_\mathrm{lc}\},\label{eq:original_equilibrium}\hspace{-6mm}\\
	&& \overline{V} - \bm{\ell}^\mathrm{T} \mathbf{a} &{}\ge{}0,\label{eq:original_volume}\\
	&& \mathbf{a} &{}\ge{}\mathbf{0},\label{eq:original_areas}
	\end{alignat}
\end{subequations}
in which $\bm{\omega} \in \mathbb{R}_{>0}^{n_\mathrm{lc}}$ are positive weights associated with $n_\mathrm{lc}$ load cases, and $\bm{\ell} \in \mathbb{R}_{\ge 0}^{n_\mathrm{e}}$ stands for the element lengths column vector. Further, $\mathbf{f}_j (\mathbf{a}) \in \mathbb{R}^{n_{\mathrm{dof},j}}$ and $\mathbf{u}_j \in \mathbb{R}^{n_{\mathrm{dof},j}}$ denote the force and displacement column vectors of the $j$-th load case, $n_{\mathrm{dof},j} \in \mathbb{N}$ stands for the associated number of degrees of freedom, and $\mathbf{K}_j(\mathbf{a}) \in \mathbb{R}^{n_{\mathrm{dof},j} \times n_{\mathrm{dof},j}}$ is the corresponding symmetric positive semidefinite stiffness matrix. For these stiffness matrices, we require $\forall \mathbf{a}>\mathbf{0}: \mathbf{K}_j(\mathbf{a}) \succ 0$ to exclude rigid body motions. Using the finite element method,
$\mathbf{K}_j(\mathbf{a})$ is assembled as
\begin{equation}\label{eq:assembly_K}
\mathbf{K}_j (\mathbf{a}) = \mathbf{K}_{j,0} + \sum_{i=1}^{n_\mathrm{e}} \left[ \mathbf{K}_{j,i}^{(1)} a_i + \mathbf{K}_{j,i}^{(2)} a_i^2 + \mathbf{K}_{j,i}^{(3)} a_i^3 \right],
\end{equation}
with $\mathbf{K}_{j,0} \succeq 0$ standing for a~design-independent stiffness (such as fixed structural elements), $\mathbf{K}_{j,i}^{(1)} \succeq 0$ being the unit-cross-section-area membrane stiffness of the $i$-th element in the $j$-th load case, and $\mathbf{K}_{j,i}^{(2)} \succeq 0$ with $\mathbf{K}_{j,i}^{(3)} \succeq 0$ are the corresponding bending stiffness counterparts associated with the unit cross-section area. The force column vector $\mathbf{f}_j$ is assumed in the form
\begin{equation}\label{eq:assembly_f}
\mathbf{f}_j (\mathbf{a}) = \mathbf{f}_{j,0} + \sum_{i=1}^{n_\mathrm{e}} \left[ \mathbf{f}_{j,i}^{(1)} a_i \right],
\end{equation}
where $\mathbf{f}_{j,0}$ stands for the design-independent load and $\mathbf{f}_{j,i}^{(1)}$ are the design-dependent loads such as self-weight. One can also add higher-order terms to \eqref{eq:assembly_f} to handle non-zero displacement boundary conditions.

The formulation \eqref{eq:original} is nonlinear and lacks convexity in general. The non-convexity comes not only from the polynomial entries in the stiffness matrix \eqref{eq:assembly_K}, but also from its possible singularity caused by zero cross-section areas \eqref{eq:original_areas}.

\subsection{Semidefinite programming formulation for topology optimization of frame structures}\label{sec:nsdp}

An approach to simplify \eqref{eq:original} relies on eliminating the displacement variables $\mathbf{u}_j$ from the problem formulation. In the nested approach, which is commonly used in topology optimization, the cross-section areas are bounded from below by a strictly positive $\varepsilon \in \mathbb{R}_{>0}$, allowing for a~computation of $\left[\mathbf{K}(\mathbf{a})\right]^{-1}$. Recall that $\mathbf{K}(\mathbf{a}) \succ 0$ for all $\mathbf{a}>\mathbf{0}$. The optimization procedure then traditionally adopts, e.g., the Method of Moving Asymptotes (MMA) \citep{Svanberg1987}, or the Optimality Criteria (OC) method \citep{Rozvany_1989}. Notice that $\varepsilon \rightarrow 0$ results in a~high condition number of $\mathbf{K}(\mathbf{a})$ and that larger values of $\varepsilon$ may impair quality of optimized designs, as sizing optimization is solved instead of the original topology optimization. In contrast, here, we eliminate $\mathbf{u}_j$ and allow the cross sections to truly attain zero.

Because $\left[\mathbf{K}(\mathbf{a})\right]^{-1}$ may not exist in our case, we rely on the Moore-Penrose pseudo-inverse $\left[\mathbf{K}(\mathbf{a})\right]^\dagger$ instead. Its role in enforcing the equilibrium conditions is clarified in the next lemma.
\begin{lemma}\label{prop:pinv}
	Consider the equation $\mathbf{K}_j(\mathbf{a}) \mathbf{u}_j = \mathbf{f}_j(\mathbf{a}) - \mathbf{r}_j$ with $\mathbf{u}_j = \mathbf{K}_j(\mathbf{a})^{\dagger} \mathbf{f}_j(\mathbf{a})$ and a residual vector $\mathbf{r}_j \in \mathbb{R}^{n_\mathrm{dof}}$. Then, $\mathbf{r}_j=\mathbf{0}$ if and only if $\mathbf{f}_j(\mathbf{a}) \in \mathrm{Im}(\mathbf{K}_j (\mathbf{a}))$.
\end{lemma}
\begin{proof}
	Let $\mathbf{f}_j (\mathbf{a}) = \mathbf{v}_j (\mathbf{a}) + \mathbf{w}_j (\mathbf{a})$, in which $\mathbf{v}_j(\mathbf{a}) \in \mathrm{Im}\left(\mathbf{K}_j(\mathbf{a})\right)$, and $\mathbf{w}_j(\mathbf{a}) \in \mathrm{Ker}\left(\mathbf{K}_j(\mathbf{a})\right)$. Then, because $\mathbf{K}_j(\mathbf{a}) \mathbf{K}_j(\mathbf{a})^{\dagger}$ is an orthogonal projector onto the range of $\mathbf{K}_j(\mathbf{a})$, we obtain $\mathbf{K}_j(\mathbf{a}) \mathbf{K}_j(\mathbf{a})^{\dagger} \mathbf{f}_j(\mathbf{a}) = \mathbf{v}_j(\mathbf{a})$. Clearly, when $\mathbf{f}_j(\mathbf{a}) \in \mathrm{Im}(\mathbf{K}_j (\mathbf{a}))$, we have $\mathbf{v}_j(\mathbf{a}) = \mathbf{f}(\mathbf{a})$ with $\mathbf{w}_j(\mathbf{a}) = \mathbf{0}$, implying that $\mathbf{r}_j = \mathbf{0}$. For the case of $\mathbf{f}_j(\mathbf{a}) \notin \mathrm{Im}(\mathbf{K}_j (\mathbf{a}))$, $\mathbf{w}_j (\mathbf{a}) \neq \mathbf{0}$, showing that $\mathbf{r}_j = -\mathbf{w}_j (\mathbf{a})$.
\end{proof}
Lemma \ref{prop:pinv} allows us to eliminate the displacement variables and write the optimization problem \eqref{eq:original} only in terms of the cross-section areas $\mathbf{a}$ as
\begin{subequations}\label{eq:fpinv}
	\begin{alignat}{2}
	\min_{\mathbf{a}}\; && \sum_{j=1}^{n_\mathrm{lc}} \omega_j \mathbf{f}_j(\mathbf{a})^\mathrm{T} \left[\mathbf{K}_j(\mathbf{a}) \right]^{\dagger} \mathbf{f}_j (\mathbf{a})\label{eq:fpinv_compliance}\\
	\mathrm{s.t.}\; && \overline{V} - \bm{\ell}^\mathrm{T} \mathbf{a} &{}\ge{} 0,\label{eq:fpinv_volume}\\
	&& \mathbf{a} &{}\ge{} \mathbf{0},\label{eq:fpinv_areas}\\
	&& \forall j \in \{1\dots n_\mathrm{lc}\}: \quad \mathbf{f}_j(\mathbf{a}) &{}\in{} \mathrm{Im}(\mathbf{K}_j(\mathbf{a})).\label{eq:fpinv_image}
	\end{alignat}
\end{subequations}
Notice that \eqref{eq:fpinv_image} essentially eliminates the nonphysical setup when $\mathbf{K}_j(\mathbf{a}) = \left[\mathbf{K}_j (\mathbf{a})\right]^\dagger = \mathbf{0}$ produces zero compliance.

Because $\left[ \mathbf{K}_j(\mathbf{a})\right]^\dagger \succeq 0, \forall j \in \{1\dots n_\mathrm{lc}\}$, and $\bm{\omega}>\mathbf{0}$ by definition, \eqref{eq:fpinv_compliance} is bounded from below by $0$. Thus, we introduce slack variables $\mathbf{c} \in \mathbb{R}_{\ge0}^{n_\mathrm{lc}}$ comprising the to-be-minimized upper bounds on compliances for the load cases and rewrite \eqref{eq:fpinv} equivalently as
\begin{subequations}\label{eq:fpinvs}
	\begin{alignat}{2}
	\min_{\mathbf{a}, \mathbf{c}}\, && \bm{\omega}^\mathrm{T} \mathbf{c}\qquad\qquad\qquad\qquad\;\;\; \label{eq:fpinvs_compliance}\\
	\mathrm{s.t.}\, && c_j - \mathbf{f}_j(\mathbf{a})^\mathrm{T} \left[\mathbf{K}_j(\mathbf{a}) \right]^{\dagger} \mathbf{f}_j (\mathbf{a}) &{}\ge{}0, \forall j \in \{1\dots n_\mathrm{lc}\},\label{eq:fpinvs_equilibrium}\hspace{-12mm}\\
	&& \overline{V} - \bm{\ell}^\mathrm{T} \mathbf{a} &{}\ge{}0,\label{eq:fpinvs_volume}\\
	&& \mathbf{a} &{}\ge{}\mathbf{0},\label{eq:fpinvs_areas}\\
	&& \forall j \in \{1\dots n_\mathrm{lc}\}: \quad \mathbf{f}_j(\mathbf{a}) &{}\in{} \mathrm{Im}(\mathbf{K}_j(\mathbf{a})).\label{eq:fpinvs_image}
	\end{alignat}
\end{subequations}

To derive a nonlinear semidefinite programming formulation, let us now recall the generalized Schur complement lemma:
\begin{lemma}\label{lemma:schur}
	\citep[Theorem 16.1]{Gallier_2011} Let $\mathbf{A}$ and $\mathbf{C}$ be symmetric square matrices, $\mathbf{B}$ have appropriate dimensions, and $\mathbf{I}$ denote an identity matrix. Then, the following conditions are equivalent:
	\begin{enumerate}
		\item $\begin{pmatrix}
		\mathbf{A} &&  \mathbf{B}^\mathrm{T}\\
		\mathbf{B} && \mathbf{C}
		\end{pmatrix} \succeq 0$,
		\item $\mathbf{C} \succeq 0$, $\mathbf{A}- \mathbf{B}^\mathrm{T} \mathbf{C}^\dagger \mathbf{B} \succeq 0$, $(\mathbf{I}-\mathbf{C}\mathbf{C}^\dagger)\mathbf{B} = \mathbf{0}$.
	\end{enumerate}
\end{lemma}
Since we already have $\mathbf{K}_j (\mathbf{a}) \succeq 0$ by definition and $c_j-\left[\mathbf{f}_j(\mathbf{a})\right]^\mathrm{T} \left[\mathbf{K}_j(\mathbf{a})\right]^\dagger \mathbf{f}_j (\mathbf{a}) \ge 0$ in \eqref{eq:fpinvs_equilibrium}, to use Lemma \ref{lemma:schur} it suffices to show that
\begin{equation}\label{eq:schur_eq}
(\mathbf{I}-\mathbf{K}_j(\mathbf{a})\left[\mathbf{K}_j(\mathbf{a})\right]^\dagger)\mathbf{f}_j (\mathbf{a}) = \mathbf{0}.
\end{equation}
\begin{proposition}\label{prop:image}
	The condition \eqref{eq:schur_eq} is equivalent to $\mathbf{f}_j(\mathbf{a}) \in \mathrm{Im}(\mathbf{K}_j(\mathbf{a}))$.
\end{proposition}
\begin{proof}
	First, consider $\mathbf{f}_j (\mathbf{a}) \in \mathrm{Im}(\mathbf{K}_j (\mathbf{a}))$. Then,  $\mathbf{f}_j (\mathbf{a}) = \mathbf{K}_j (\mathbf{a}) \mathbf{u}_j$ for some displacement vector $\mathbf{u}_j$. After inserting it into the left-hand-side of \eqref{eq:schur_eq}, we have
	\begin{equation}
	\left( \mathbf{K}_j(\mathbf{a}) - \mathbf{K}_j(\mathbf{a}) \left[\mathbf{K}_j(\mathbf{a})\right]^\dagger \mathbf{K}_j(\mathbf{a}) \right) \mathbf{u}_j = \mathbf{0},
	\end{equation}
	which holds for all such $\mathbf{u}_j$ as $\mathbf{K}_j(\mathbf{a}) \left[\mathbf{K}_j(\mathbf{a})\right]^\dagger \mathbf{K}_j(\mathbf{a}) = \mathbf{K}_j(\mathbf{a})$ by the definition of the Moore-Penrose pseudo-inverse \citep[Lemma 14.1]{Gallier_2011}.
	
	Otherwise, consider $\mathbf{f}_j (\mathbf{a}) \notin \mathrm{Im}(\mathbf{K}_j (\mathbf{a}))$ and let $\tilde{\mathbf{u}} = \left[\mathbf{K}_j (\mathbf{a})\right]^\dagger \mathbf{f}_j (\mathbf{a})$. Then, $\mathbf{K}_j (\mathbf{a}) \tilde{\mathbf{u}} = \mathbf{f}_j (\mathbf{a}) - \mathbf{r}_j$ for some $\mathbf{r}_j \in \mathrm{Ker}(\mathbf{K}(\mathbf{a}))$, $\mathbf{r}_j \neq \mathbf{0}$ by Lemma \ref{prop:pinv}. Thus, the left-hand-side of \eqref{eq:schur_eq} simplifies to
	\begin{equation}
	\begin{multlined}
	\mathbf{f}_j (\mathbf{a}) - \mathbf{K}_j (\mathbf{a}) \left[\mathbf{K}_j (\mathbf{a})\right]^{\dagger} \mathbf{f}_j (\mathbf{a}) = \\
	=\mathbf{f}_j (\mathbf{a}) -\mathbf{K}_j (\mathbf{a})\tilde{\mathbf{u}} = \mathbf{r}_j \neq \mathbf{0},
	\end{multlined}
	\end{equation}
	which completes the proof. 
\end{proof}
Finally, Proposition \ref{prop:image} and Lemma \ref{lemma:schur} facilitate an equivalent reformulation of the optimization problem \eqref{eq:fpinvs} as a~nonlinear semidefinite program
\begin{subequations}\label{eq:nsdp}
	\begin{alignat}{3}
	& \min_{\mathbf{a}, \mathbf{c}} \;& \bm{\omega}^\mathrm{T} \mathbf{c} \qquad\qquad\qquad\label{eq:nsdp_obj}\\
	& \;\mathrm{s.t.}\;\; &
	\begin{pmatrix}
	c_j & -\mathbf{f}_j(\mathbf{a})^\mathrm{T}\\
	-\mathbf{f}_j(\mathbf{a}) & \mathbf{K}_j(\mathbf{a})
	\end{pmatrix}&{}\succeq{} 0,\; \forall j \in \{1\dots n_\mathrm{lc}\},\hspace{-4mm}\label{eq:pmi}\\
	&& \overline{V} - \bm{\ell}^\mathrm{T} \mathbf{a} &{}\ge{}0,\label{eq:nsdp_vol}\\
	&& \mathbf{a} &{}\ge{} \mathbf{0},\label{eq:nsdp_areas}
	\end{alignat}
\end{subequations}
in which only the constraint \eqref{eq:pmi} lacks convexity. Importantly, all constraints are polynomial functions of $\mathbf{a}$, forming therefore a semi-algebraic feasible set.

\subsection{Efficient polynomial reformulation}\label{sec:eff}

The optimization problem \eqref{eq:nsdp} constitutes a minimization of a linear function over a semi-algebraic set, allowing for a~solution using the moment-sum-of-squares hierarchy, as briefly discussed in Section \ref{sec:mom-sos}. However, efficiency of the hierarchy can be improved after modifying \eqref{eq:nsdp} to provide a~tighter feasible set of relaxed problems and to reduce numerical issues by scaling the design variables. These modifications are outlined in the following paragraphs.

\subsubsection{Compactness of the feasible set}

We start by enforcing compactness of the feasible set of the optimization problem \eqref{eq:nsdp} because of two reasons. First, compactness is required for Theorem \ref{th:convergence}, in the form of Assumption~\ref{ass:comp}. Second, compactness also allows tightening the feasible sets of relaxed problems, notably improving numerical performance.

\begin{proposition}\label{prop:bounds}
	Assume that $\mathbf{a}^*$ and $\mathbf{c}^*$ are optimal cross-section areas and compliances associated with the optimization problem \eqref{eq:nsdp}. Then, $\forall i \in \{1\dots n_e\}: 0 \le a_i^* \le \overline{a}_i$ with $\overline{a}_i = \overline{V}/\ell_i$ and $\forall j \in \{1\dots n_\mathrm{lc}\}: 0 \le c_j^* \le \overline{c}/\omega_j$, where $\overline{c} = \sum_{j=1}^{n_\mathrm{lc}} \left[ \omega_j \mathbf{f}(\hat{\mathbf{a}})^\mathrm{T} \mathbf{K}(\hat{\mathbf{a}})^{-1} \mathbf{f}(\hat{\mathbf{a}}) \right]$ with  $\hat{\mathbf{a}} = \mathbf{1} \overline{V}/\sum_{i=1}^{n_\mathrm{e}} \ell_i$.
\end{proposition}
\begin{proof}
	The cross-section areas are non-negative by definition \eqref{eq:nsdp_areas}. Therefore, \eqref{eq:nsdp_vol} represents a conic combination and none of the structural elements can occupy a~larger volume than the volume bound $\overline{V}$, $\forall i \in \{1\dots n_\mathrm{e}\}: a_i^* \le \overline{V}/\ell_i$.
	
	The compliance variables are placed at the main diagonal of the polynomial matrix inequality (PMI) \eqref{eq:pmi} and are hence non-negative, $c_j^* \ge 0$. Then, because $\bm{\omega} > \mathbf{0}$, the conic combination $\bm{\omega}^\mathrm{T} \mathbf{c}^*$ is an upper bound for its summands, $\omega_j c_j^* \le \bm{\omega}^\mathrm{T} \mathbf{c}^*$. Moreover, since $\hat{\mathbf{a}}$ determines uniquely the compliances $\hat{\mathbf{c}}$, $\hat{c}_j = \mathbf{f}_j (\hat{\mathbf{a}})^\mathrm{T}\mathbf{K}_j(\hat{\mathbf{a}})^{-1}\mathbf{f}_j(\hat{\mathbf{a}})$, the pair $(\hat{\mathbf{a}}, \hat{\mathbf{c}})$ is a feasible solution to \eqref{eq:pmi}--\eqref{eq:nsdp_areas}, so that we also have $\bm{\omega}^\mathrm{T} \mathbf{c}^* \le \overline{c} = \bm{\omega}^\mathrm{T} \hat{\mathbf{c}}$. Consequently, $\forall j \in \{1\dots n_\mathrm{lc}\}: \omega_j c_j^* \le \overline{c}$.
\end{proof}

Among the bounds in Proposition \ref{prop:bounds}, only the compliance upper bounds are not enforced in the formulation \eqref{eq:nsdp}. Indeed, for any fixed $\mathbf{a}>\mathbf{0}$, $\mathbf{c} \rightarrow \bm{\infty}$ is feasible to \eqref{eq:nsdp}, so that Assumption \ref{ass:comp} is not satisfied. To make the design space bounded, we add the (redundant) upper-bound compliance constraint from Proposition \ref{prop:bounds}, or, eventually, an upper-bound obtained by solving the convex truss topology optimization problem instead, see Appendix~\ref{app:tto}. Subsequently, we arrive at the optimization problem
\begin{subequations}\label{eq:nsdpC}
	\begin{alignat}{2}
	\min_{\mathbf{a}, \mathbf{c}}\; && \bm{\omega}^\mathrm{T} \mathbf{c} \qquad\qquad\qquad\label{eq:nsdpC_obj}\\
	\mathrm{s.t.}\; && 
	\begin{pmatrix}
	c_j & -\mathbf{f}_j(\mathbf{a})^\mathrm{T}\\
	-\mathbf{f}_j(\mathbf{a}) & \mathbf{K}_j(\mathbf{a})
	\end{pmatrix}&{}\succeq{} 0,\; \forall j \in \{1\dots n_\mathrm{lc}\},\label{eq:ndspC_pmi}\hspace{-3mm}\\
	&& \overline{V}-\bm{\ell}^\mathrm{T} \mathbf{a} &{}\ge{}0,\label{eq:nsdpC_vol}\\
	&& \overline{c}-\bm{\omega}^\mathrm{T} \mathbf{c} &{}\ge{}0,\\
	&& \mathbf{a}&{}\ge{}\mathbf{0},\label{eq:nsdpC_areas}
	\end{alignat}
\end{subequations}
for which we have the following result:

\begin{proposition}\label{prop:compact}
	The feasible set of \eqref{eq:nsdpC} is compact.
\end{proposition}
\begin{proof}
	The feasible set is bounded based on Proposition \ref{prop:bounds}. Moreover, $\mathbf{a}$ and $\mathbf{c}$ satisfying conditions \eqref{eq:nsdpC_vol}--\eqref{eq:nsdpC_areas} form a~closed set. Thus, it suffices to show that \eqref{eq:ndspC_pmi} is closed. But the elements in \eqref{eq:ndspC_pmi} are polynomial functions that are continuous. Moreover, the set of semidefinite matrices is closed so $\mathbf{a}$ and $\mathbf{c}$ satisfying \eqref{eq:ndspC_pmi} live in a closed set. Boundedness and closeness imply compactness because we are in a finite dimensional space.
\end{proof}

\subsubsection{Scaling and box constraints}

After introducing box constraints in formulation \eqref{eq:nsdpC}, we can scale all variables domains to $\left[-1,1\right]$. This scaling reduces numerical issues that may arise during the problem solution. To this goal, we have
\begin{subequations}
	\begin{align}
	&c_j = 
	\frac{1}{2 \omega_j} \left(c_{\mathrm{s},j} + 1\right) \overline{c}, &\forall j &\in \{1\dots n_\mathrm{lc}\},\label{eq:scaled_c}\\
	&a_i = 0.5 \left(a_{\mathrm{s},i} + 1\right) \overline{a}_i,\quad &\forall i &\in \{1\dots n_\mathrm{e}\},\label{eq:scaled_a}
\end{align}
\end{subequations}
where $\mathbf{a}_\mathrm{s}$ and $\mathbf{c}_\mathrm{s}$ are the scaled cross-section areas and compliance variables.

In addition, we explicitly insert the box constraints into the optimization problem formulation to tighten feasible sets of the relaxed problems. There are multiple options how to write these box constraints $\mathbf{a}_\mathrm{s}, \mathbf{c}_\mathrm{s} \in \left[-1,1\right]$, e.g.,
\begin{subequations}
	\begin{align}
	\begin{aligned}\label{eq:csc_lin}
	-1 \le a_{\mathrm{s},i} &\le 1, \quad \forall i \in \{1\dots n_\mathrm{b}\},\\
	-1 \le c_{\mathrm{s},j} &\le 1, \quad \forall j \in \{1\dots n_\mathrm{lc}\},
	\end{aligned}\\
	\begin{aligned}\label{eq:csc_all}
	a_{\mathrm{s},i}^2 &\le 1, \quad \forall i \in \{1\dots n_\mathrm{b}\},\\
	c_{\mathrm{s},j}^2 &\le 1, \quad \forall j \in \{1\dots n_\mathrm{lc}\}.
	\end{aligned}\\
	\begin{aligned}\label{eq:csc_all2}
	a_{\mathrm{s},i}^4 &\le 1, \quad \forall i \in \{1\dots n_\mathrm{b}\},\\
	c_{\mathrm{s},j}^4 &\le 1, \quad \forall j \in \{1\dots n_\mathrm{lc}\}.
	\end{aligned}
	\end{align}
\end{subequations}
Despite equivalent in what they enforce, their numerical performance in the moment-sum-of-squares hieararchy varies considerably; we refer the reader to the recent note of \citet{Anjos2020}. Here, we use the quadratic bounds \eqref{eq:csc_all} when the third-order terms $\mathbf{K}_{i,j}^{(3)} a_i^3$ are absent, and quartic bounds \eqref{eq:csc_all2} otherwise. Then, the optimization problem reads
\begin{subequations}\label{eq:nsdpSC}
	\begin{alignat}{2}
	\min_{\mathbf{a}_\mathrm{s}, \mathbf{c}_\mathrm{s}}\; &&\sum_{j=1}^{n_\mathrm{lc}} \left[
	0.5 \left(c_{\mathrm{s},j} + 1\right) \overline{c}\right]\qquad\qquad\quad\;\;
	\label{nsdpSC_obj}\hspace{-6mm}\\
	\mathrm{s.t.}\; &&
	\begin{alignedat}{1}
	\begin{pmatrix}
	\frac{1}{2 \omega_j} \left(c_{\mathrm{s},j} + 1\right) \overline{c} & -\mathbf{f}_j(\mathbf{a}_\mathrm{s})^\mathrm{T}\\
	-\mathbf{f}_j(\mathbf{a}_\mathrm{s}) & \mathbf{K}_j(\mathbf{a}_\mathrm{s})
	\end{pmatrix}\succeq{} 0,\\
	\forall j \in \{1\dots n_\mathrm{lc}\},
	\end{alignedat}
	\label{eq:pmiSC}\hspace{-6mm}\\
	&& 2 - n_\mathrm{e} - \mathbf{1}^\mathrm{T} \mathbf{a}_{\mathrm{s}} {}\ge{}0,\label{eq:volSC}\hspace{-6mm}\\
	&& 1 - \bm{\omega}^\mathrm{T} \mathbf{c}_\mathrm{s} {}\ge{}0,\label{eq:compSumSC}\hspace{-6mm}\\
	&& \text{bound constraints } \eqref{eq:csc_all} \text{ or } \eqref{eq:csc_all2}.\hspace{-6mm}\label{eq:boxSC}
	\end{alignat}
\end{subequations}
Because the feasible set of \eqref{eq:nsdpSC} is compact by Proposition \ref{prop:compact}, one may tempt to add a redundant polynomial inequality constraint to satisfy Assumption \ref{ass:comp}. However, the assumption is already satisfied in our case.
\begin{proposition}\label{prop:archimedean}
	The optimization problem \eqref{eq:nsdpSC} satisfies Assumption \ref{ass:comp}. 
\end{proposition}
\begin{proof}
	Let $\mathbf{G}(\mathbf{a}_\mathrm{s},\mathbf{c}_\mathrm{s})$ be a block-diagonal matrix with the blocks \eqref{eq:pmiSC}--\eqref{eq:boxSC} and let $\mathbf{H}$ be a sparse matrix of the same dimensions with the structure
	\begin{equation}
	\mathbf{H} = \begin{pmatrix}
	\mathbf{0} & \mathbf{0}\\
	\mathbf{0} & \mathbf{I}
	\end{pmatrix},
	\end{equation}
	in which the identity matrix $\mathbf{I} \in \mathbb{S}^{n_\mathrm{e} + n_\mathrm{lc}}$ matches the positions of \eqref{eq:boxSC} in $\mathbf{G}(\mathbf{a}_\mathrm{s},\mathbf{c}_\mathrm{s})$. Clearly, $\mathbf{H}$ is a SOS because of $\mathbf{H} = \mathbf{H} \mathbf{H}^\mathrm{T}$, recall Definition \ref{def:sos}. Then, if \eqref{eq:csc_all} is used, $p_{\eqref{eq:csc_all}} = \langle \mathbf{H}\mathbf{H}^\mathrm{T}, \mathbf{G}(\mathbf{a}_\mathrm{s},\mathbf{c}_\mathrm{s}) \rangle = n_\mathrm{e}+n_\mathrm{lc} - \sum_{i=1}^{n_\mathrm{e}}a_{\mathrm{s},i}^2 - \sum_{j=1}^{n_\mathrm{lc}}c_{\mathrm{s},j}^2$, so that the level set $\{\mathbf{a}_\mathrm{s} \in \mathbb{R}^{n_\mathrm{e}}, \mathbf{c}_\mathrm{s} \in \mathbb{R}^{n_\mathrm{lc}}\;\vert\; p_{\eqref{eq:csc_all}} \ge 0\}$ is compact. For \eqref{eq:csc_all2}, $p_{\eqref{eq:csc_all2}} = n_\mathrm{e}+n_\mathrm{lc} - \sum_{i=1}^{n_\mathrm{e}}a_{\mathrm{s},i}^4 - \sum_{j=1}^{n_\mathrm{lc}}c_{\mathrm{s},j}^4$, showing that also the level set $\{\mathbf{a}_\mathrm{s} \in \mathbb{R}^{n_\mathrm{e}}, \mathbf{c}_\mathrm{s} \in \mathbb{R}^{n_\mathrm{lc}}\;\vert\; p_{\eqref{eq:csc_all2}} \ge 0\}$ is compact.
\end{proof}
\begin{remark}\label{rem:tigher}
The higher-order constraints in \eqref{eq:boxSC} are tighter in the moment representation, i.e., \eqref{eq:csc_all} are tighter than \eqref{eq:csc_lin} and \eqref{eq:csc_all2} are tighter than \eqref{eq:csc_all}.
\end{remark}
To see this, assume that $\left(y_0, y_1, y_2, y_3, y_4\right)$ are the moments associated with the canonical basis of the vector space of polynomials of degree at most four, $(1, x, x^2, x^3, x^4)$, where one can substitute $x$ by any element of $\mathbf{a}_\mathrm{s}$ or $\mathbf{c}_\mathrm{s}$. Then, in the first relaxation of the moment-sum-of-squares hierarchy, the fourth-order constraint $1-x^4 \ge 0$ becomes
\begin{equation}\label{eq:4}
y_0 - y_4 \ge 0,
\end{equation}
the quadratic constraint $1-x^2 \ge 0$ yields
\begin{equation}
y_0 - y_{2} \ge 0, \label{eq:1}
\end{equation}
and the box constraint $-1 \le x \le 1$ provides
\begin{equation}
-y_0 \le y_1 \le y_0, \label{eq:box}
\end{equation}
with $y_0 = 1$. Moreover, the localizing matrix of the entire optimization problem contains principal submatrices
\begin{subequations}\label{eq:lmat}
\begin{align}
\begin{pmatrix}
y_0  & y_1 \\
y_1 & y_2
\end{pmatrix} &\succeq 0,\label{eq:lmat_a}\\
\begin{pmatrix}
y_0  & y_2 \\
y_2 & y_4
\end{pmatrix} &\succeq 0,\label{eq:lmat_b}\\
\begin{pmatrix}
y_2  & y_3 \\
y_3 & y_4
\end{pmatrix} &\succeq 0,\label{eq:lmat_c}
\end{align}
\end{subequations}
that must be positive semi-definite as the entire localizing matrix is. Notice that \eqref{eq:lmat_b} and \eqref{eq:lmat_c} are not present in the degree-one relaxation, which can, therefore, be used only if the bending stiffness is a polynomial of degree at most two. Thus, we omit them in the reasoning behind the relaxation tightness for \eqref{eq:csc_lin} and \eqref{eq:csc_all}.

In the lowest, degree-two, relaxation the fourth-order constraints \eqref{eq:csc_all2} provide $y_4 \le 1$ from Eq. \eqref{eq:4} and $y_2 \ge 0$ with $y_4 \ge 0$ based on \eqref{eq:lmat}. Moreover, the determinants of \eqref{eq:lmat} must be non-negative, implying that $y_1^2 \le y_2$, $y_3^2 \le y_2 y_4$, and $y_2^2 \le y_4$. A combination of these inequalities then results in $0 \le y_1^4 \le y_2^2 \le y_4 \le 1$ and $0 \le y_3^2 \le y_2 y_4 \le 1$.

For the quadratic constraints \eqref{eq:csc_all}, $y_2 \le 1$ from Eq. \eqref{eq:1} and $y_2 \ge 0$ because of Eq. \eqref{eq:lmat}. Writing the determinant of \eqref{eq:lmat} then provides us with $y_1^2 \le y_2$. Consequently, we observe that $0 \le y_1^2 \le y_2 \le 1$. Notice that this inequality is automatically satisfied in the preceding case.

In the case of pure box constraints \eqref{eq:csc_lin}, we only have $0 \le y_1^2 \le 1$, Eq. \eqref{eq:box}, and $y_1^2 \le y_2$, Eq.~\eqref{eq:lmat}. Note that there is no upper bound for $y_2$, which can attain arbitrarily large values in the first relaxation. From the mechanical point of view, this allows for an arbitrarily-large rotational stiffnesses $\mathbf{K}_{j,i}^{(2)} a_i^2$ of the elements.

These observations then allow us to show feasibility of the first-order moments for \eqref{eq:volSC}--\eqref{eq:boxSC}:

\begin{proposition}\label{prop:feasiblemom}
	Let $\mathbf{y}^*_{\mathbf{c}^1}$ and $\mathbf{y}^*_{\mathbf{a}^1}$ be the first-order moments associated with the variables $\mathbf{c}_\mathrm{s}$ and $\mathbf{a}_\mathrm{s}$ obtained from a solution to any relaxation of \eqref{eq:nsdpSC} using the moment-sum-of-squares hierarchy. Then, these moments satisfy
	\begin{subequations}
	\begin{align}
	2 - n_\mathrm{e} - \mathbf{1}^\mathrm{T} \mathbf{y}^*_{\mathbf{a}^1} &\ge 0,\label{eq:fom1}\\
	1 - \bm{\omega}^\mathrm{T} \mathbf{y}^*_{\mathbf{c}^1} &\ge 0,\label{eq:fom2}\\
	1 - \left(y^*_{a_i^1}\right)^2 & \ge 0 \quad\text{if \eqref{eq:csc_all} is used},\label{eq:fom3}\\
	1 - \left(y^*_{a_i^1}\right)^4 & \ge 0 \quad\text{if \eqref{eq:csc_all2} is used},\\
	1 - \left(y^*_{c_j^1}\right)^2 & \ge 0 \quad\text{if \eqref{eq:csc_all} is used},\\
	1 - \left(y^*_{c_j^1}\right)^4 & \ge 0 \quad\text{if \eqref{eq:csc_all2} is used}.\label{eq:fom4}
	\end{align}
	\end{subequations}
\end{proposition}
\begin{proof}
	\eqref{eq:fom1} and \eqref{eq:fom2} hold trivially from construction of the hierarchy. \eqref{eq:fom3}--\eqref{eq:fom4} follow from Remark \ref{rem:tigher}.
\end{proof}

\subsection{Recovering feasible upper-bound solutions}\label{sec:ub}

In Proposition \ref{prop:feasiblemom}, we have shown that the first-order moments obtained by solving any relaxation of the moment-sum-of-squares hierarchy satisfy all the constraints of \eqref{eq:nsdpSC} except for \eqref{eq:pmiSC}. This section is therefore devoted to the question how to ``correct'' these moments to produce feasible upper-bounds to the original problem \eqref{eq:nsdp} and provide a natural sufficient condition of global optimality.

We start by proving the following essential result:

\begin{proposition}\label{prop:inimage}
	Let $\mathbf{y}^*_{\mathbf{c}^1}$ and $\mathbf{y}^*_{\mathbf{a}^1}$ be the first-order moments associated with the variables $\mathbf{c}_\mathrm{s}$ and $\mathbf{a}_\mathrm{s}$ obtained from a solution to any relaxation of \eqref{eq:nsdpSC} using the moment-sum-of-squares hierarchy and let $\forall i \in \{1\dots n_\mathrm{e}\}: \tilde{a}_i = 0.5(y_{a_i^1}^{*}+1)\overline{a}$ be the corresponding cross-section areas. Then,
	\begin{equation}\label{eq:propinimage}
	\mathbf{f}_{j,0} + \sum_{i=1}^{n_\mathrm{e}} \mathbf{f}_{j,i} \tilde{a}_i \in \mathrm{Im}\left(\mathbf{K}_{j,0} + \sum_{i=1}^{n_\mathrm{e}}\sum_{k=1}^3 \mathbf{K}_{j,i}^{(k)} \tilde{a}_i^k\right).
	\end{equation}
\end{proposition}
\begin{proof}
	In the lowest relaxation of the moment-sum-of-squares hierarchy, the PMI constraint \eqref{eq:pmiSC} becomes
	\begin{equation}
	\begin{pmatrix}
	\frac{1}{2 \omega_j} \left(y_{c_j^1}^{*} + 1\right) \overline{c} & -\mathbf{f}_j^\mathrm{T}\left(\mathbf{y}^{*}_{\mathbf{a}^1}\right)^\mathrm{T}\\
	-\mathbf{f}_j\left(\mathbf{y}^{*}_{\mathbf{a}^1}\right) & \mathbf{K}_j\left(\mathbf{y}^{*}_{\mathbf{a}^1}, \mathbf{y}^{*}_{\mathbf{a}^2}, \mathbf{y}^{*}_{\mathbf{a}^3}\right)
	\end{pmatrix} \succeq 0,
	\end{equation}
	where $\mathbf{y}^*_{\mathbf{a}^2}$ and $\mathbf{y}^*_{\mathbf{a}^3}$ are the second- and third-order moments associated with $\mathbf{a}_\mathrm{s}$ and, with a slight abuse of notation, $\mathbf{K}_j\left(\mathbf{y}^{*}_{\mathbf{a}^1}, \mathbf{y}^{*}_{\mathbf{a}^2}, \mathbf{y}^{*}_{\mathbf{a}^3}\right)$ and $\mathbf{f}_j  \left(\mathbf{y}^{*}_{\mathbf{a}^1}\right)$ are the stiffness matrix and force column vector constructed from the moments $\mathbf{y}$. Using Lemma \ref{lemma:schur} and Proposition \ref{prop:image}, we observe that 
	\begin{equation}
	\mathbf{f}(\mathbf{y}^{*}_{\mathbf{a}^1}) \in \mathrm{Im}\left(\mathbf{K}_j\left(\mathbf{y}^{*}_{\mathbf{a}^1}, \mathbf{y}^{*}_{\mathbf{a}^2}, \mathbf{y}^{*}_{\mathbf{a}^3}\right)\right).
	\end{equation}
	Because we have considered solely degree-one moments in \eqref{eq:propinimage} and $\forall \mathbf{a}>\mathbf{0}: \mathbf{K}_j(\mathbf{a}) \succ 0$ was our initial assumption, we must show that the combination of $a_i = 0$ with $I_i > 0$ cannot occur for any $i$, because that would result in a lower rank of $\mathbf{K}_j (\tilde{\mathbf{a}})$ when compared with $\mathbf{K}_j\left(\mathbf{y}^{*}_{\mathbf{a}^1}, \mathbf{y}^{*}_{\mathbf{a}^2}, \mathbf{y}^{*}_{\mathbf{a}^3}\right)$.
	
	To this goal, let $a_i = 0$, which is equivalent to $y_{a_i^1}^{*} = -1$. Then, the non-negative determinant of \eqref{eq:lmat_a} and \eqref{eq:lmat_b} with the inequalities \eqref{eq:csc_all} or \eqref{eq:csc_all2} imply that $y_{a_i^2}^{*} = y_{a_i^4}^{*} = 1$. Moreover, also the determinant of the principal submatrix
	\begin{equation}\label{eq:mom2}
	\begin{pmatrix}
	1 & y_{a_i^1}^{*} & y_{a_i^2}^{*}\\
	y_{a_i^1}^{*} & y_{a_i^2}^{*} & y_{a_i^3}^{*}\\
	y_{a_i^2}^{*} & y_{a_i^3}^{*} & y_{a_i^4}^{*}
	\end{pmatrix} \succeq 0
	\end{equation}
	of the moment matrix must be non-negative. Inserting $y_{a_i^2}^{*} = y_{a_i^4}^{*} = 1$ with $y_{a_i^1}^{*}=-1$ into \eqref{eq:mom2} yields a unique feasible $y_{a_i^3}^{*} = -1$. Thus, we write the moment of inertia in terms of the scaled cross-section areas \eqref{eq:scaled_a}. After inserting first-order moments, we obtain
	\begin{multline}
	I_i = 0.25 c_\mathrm{II} \overline{a}^2 \left( y_{a_i^2}^{*} + 2 y_{a_i^1}^{*} + 1\right) + \\
	0.125 c_\mathrm{III} \overline{a}^3 \left(y_{a_i^3}^{*} + 3 y_{a_i^2}^{*} + 3 y_{a_i^1}^{*} + 1\right) = 0.
	\end{multline}
\end{proof}

Using Proposition \ref{prop:inimage}, we can correct $\mathbf{c}$ based on $\mathbf{y}_{\mathbf{a}^1}^{*}$ to provide a feasible solution to \eqref{eq:nsdp}.

\begin{theorem}\label{th:feasible}
	Let $\mathbf{y}_{\mathbf{c}^1}^{*}$ and $\mathbf{y}_{\mathbf{a}^1}^{*}$ be the first-order moments associated with the variables $\mathbf{c}_\mathrm{s}$ and $\mathbf{a}_\mathrm{s}$ obtained from a solution to any relaxation of \eqref{eq:nsdpSC} using the moment-sum-of-squares hierarchy. Then,
	\begin{subequations}
	\begin{align}
	\tilde{a}_i &= 0.5 (y_{a_i^1}^{*} + 1) \overline{a}_i,\;&\forall i \in \{1\dots n_\mathrm{e}\},\\
	\tilde{c}_j &= \left[\mathbf{f}_j (\tilde{\mathbf{a}})\right]^\mathrm{T} \mathbf{K}_j^{\dagger}(\tilde{\mathbf{a}}) \mathbf{f}_j(\tilde{\mathbf{a}}), \;&\forall j \in \{1\dots n_\mathrm{lc}\}\label{eq:corrcompl}
	\end{align}
	\end{subequations}
	is feasible (upper-bound) to \eqref{eq:nsdp}.
\end{theorem}
\begin{proof}
	Based on Proposition \ref{prop:feasiblemom}, $\tilde{\mathbf{a}}$ satisfies the constraints imposed on the cross-section areas. By correcting the compliance variables according to \eqref{eq:corrcompl}, the equilibrium equation (and so the PMI \eqref{eq:pmi}) is satisfied due to Proposition \ref{prop:inimage}. Consequently, all the constraints of \eqref{eq:nsdp} are feasible for the pair $\tilde{\mathbf{a}}$, $\tilde{\mathbf{c}}$, showing that $\bm{\omega}^\mathrm{T} \mathbf{c}^* \le \bm{\omega}^\mathrm{T} \tilde{\mathbf{c}} < \infty$.
\end{proof}

We wish to emphasize that in Theorem \ref{th:feasible}, we have proved feasibility of the upper bounds to \eqref{eq:nsdp} and such upper bounds may violate the compliance bound constraints \eqref{eq:compSumSC}. Thus, knowledge of $\bm{\omega}^\mathrm{T} \mathbf{c}^*$ does not assure convergence of the lowest relaxation to the optimal cross-section areas.

\subsection{Certificate of global $\varepsilon$-optimality}\label{sec:opt}

Because the hierarchy generates a sequence of lower bounds and we have just shown in Theorem \ref{th:feasible} how to compute upper bounds in each relaxation, we naturally arrive at a simple sufficient condition of global $\varepsilon$-optimality.

\begin{theorem}\label{th:suff}
	Let $\mathbf{y}_{\mathbf{c}^1}^{*}$ and $\mathbf{y}_{\mathbf{a}^1}^{*}$ be the first-order moments associated with the variables $\mathbf{c}_\mathrm{s}$ and $\mathbf{a}_\mathrm{s}$ obtained from a solution to any relaxation of \eqref{eq:nsdpSC} using the moment-sum-of-squares hierarchy. Then,
	\begin{equation}\label{eq:suff}
	\bm{\omega}^\mathrm{T}\tilde{\mathbf{c}} - \sum_{j=1}^{n_\mathrm{lc}}\left[ 0.5 (y_{c_j^1}^{*}+1) \overline{c} \right] \le \varepsilon
	\end{equation}
	is a sufficient condition of global $\varepsilon$-optimality.
\end{theorem}
Theorem \ref{th:suff} is very simple to verify computationally, significantly simpler than the traditional rank-based certificate of global optimality \eqref{eq:rank}, e.g., \citep{Henrion_2006}. However, \eqref{eq:suff} fails to be a necessary condition. Indeed, the optimality gap $\varepsilon$ may remain strictly positive even when the hierarchy converged according to \eqref{eq:rank} in the case of multiple globally optimal solutions. Then, the optimal first-order moments $\mathbf{y}$ are not unique; for instance, they may correspond to any convex combination of the global optima, we refer to Section \ref{sec:failed} for a specific example. 

A stronger result holds, however, when the optimization problem possesses a~unique global optimum. To show this, we first prove that, with an increasing relaxation degree $r$, the feasible space of relaxations converges to the convex hull of the initial (non-convex) problem.
\begin{proposition}\label{prop:hull}
	Let $\mathcal{K}^{(r)}$ be the feasible set of the first-order moments in the $r$-th relaxation of the moment-sum-of-squares hierarchy of \eqref{eq:nsdpSC}. Then, $\mathcal{K}^{(r)} \uparrow \text{conv}(\mathcal{K})$ as $r\rightarrow \infty$, where $\mathcal{K}$ is the intersection of \eqref{eq:pmiSC}--\eqref{eq:boxSC}.
\end{proposition}
\begin{proof}
	Let $f(\mathbf{a}_\mathrm{s},\mathbf{c}_\mathrm{s})$ be an arbitrary affine function. Based on Proposition \ref{prop:archimedean}, Assumption \ref{ass:comp} holds for \eqref{eq:nsdpSC} independently of the objective function. Hence, optimization of $f(\mathbf{a}_\mathrm{s}, \mathbf{c}_\mathrm{s})$ over $\mathcal{K}$ yields $f(\mathbf{a}_\mathrm{s},\mathbf{c}_\mathrm{s}) \uparrow f^*(\mathbf{a}_\mathrm{s}, \mathbf{c}_\mathrm{s})$ as $r\rightarrow \infty$ due to Theorem \ref{th:convergence}. Because $f(\mathbf{a}_\mathrm{s}, \mathbf{c}_\mathrm{s})$ is arbitrary, $\mathcal{K}^{(r)} \uparrow \text{conv}(\mathcal{K})$ as $r\rightarrow \infty$.
\end{proof}

Finally, we can prove that the hierarchy eventually attains a~zero optimality gap.
\begin{theorem}\label{th:zero}
	If there is a unique global solution to \eqref{eq:nsdpSC}, then
	\begin{equation}
	\bm{\omega}^\mathrm{T}\tilde{\mathbf{c}} - \sum_{j=1}^{n_\mathrm{lc}}\left[ 0.5 (y_{c_j^1}^{(r)*}+1) \overline{c} \right] = 0
	\end{equation}
	as $r\rightarrow \infty$.
\end{theorem}
\begin{proof}
	Assuming $r \rightarrow \infty$, optimization of \eqref{nsdpSC_obj} over $\mathcal{K}^{(r)}$ is equivalent to optimization of \eqref{nsdpSC_obj} over $\text{conv}(\mathcal{K})$ by Proposition \ref{prop:hull}. Because $\mathcal{K}$ is compact, its convex hull must be also compact. Hence, it can be equivalently expressed as the convex hull of the limit points of $\mathcal{K}$ that are denoted by $\mathbf{d}_1, \mathbf{d}_2, \dots$, i.e.,
	\begin{equation}
	\text{conv}(\mathcal{K}) = \text{conv}(\cup_{i=1}^{\infty} \{\mathbf{d}_i\}).
	\end{equation}
	Because we assume there is the unique global optimum when optimizing over $\mathcal{K}$, there must be a unique limit point $\mathbf{d}^*$ associated with this optimum.
\end{proof}

\begin{remark}
	Although Theorem \ref{th:zero} relies on $r \rightarrow \infty$, a finite (and fairly small) $r$ is required in all our test cases to reach the zero optimality gap. Moreover, this bound equality has occurred when the hierarchy converged based on the rank test \eqref{eq:rank}. It might be possible, therefore, to strengthen Theorem \ref{th:zero} to a finite termination result.
\end{remark}

\subsection{Global topology optimization of shell structures}\label{sec:shells}

Until now, solely frame structures have been considered. However, the optimization formulations \eqref{eq:nsdp} and \eqref{eq:nsdpSC} allow for simple modifications to optimize other discrete structures such as shells. Let $\mathbf{t} \in \mathbb{R}_{\ge 0}^{n_\mathrm{e}}$ be the vector of shell element thicknesses. Then, the formulation \eqref{eq:nsdp} becomes
\begin{subequations}\label{eq:nsdpSH}
	\begin{alignat}{5}
	\min_{\mathbf{t}, \mathbf{c}}\; && \bm{\omega}^\mathrm{T} \mathbf{c} \qquad\qquad\qquad\;\;\label{eq:nsdp_objSH}\\
	\mathrm{s.t.}\;\; && 
	\begin{pmatrix}
	c_j & -\mathbf{f}_j(\mathbf{t})^\mathrm{T}\\
	-\mathbf{f}_j(\mathbf{t}) & \mathbf{K}_j(\mathbf{t})
	\end{pmatrix}\;&&\succeq&&\; 0,&&\; \forall j \in \{1\dots n_\mathrm{lc}\},\hspace{-4mm}\label{eq:pmiSH}\\
	&& \overline{V} - \mathbf{s}^\mathrm{T} \mathbf{t} \;&&\ge&&\;0,\label{eq:nsdp_volSH}\\
	&& \mathbf{t}\;&&\ge&&\; \mathbf{0},\label{eq:nsdp_areasSH}
	\end{alignat}
\end{subequations}
where $\mathbf{s} \in \mathbb{R}_{>0}^{n_\mathrm{e}}$ is a vector of the surface areas of individual shell elements, and $\mathbf{K}_j(\mathbf{t})$ is assembled as
\begin{equation}
\mathbf{K}_j (\mathbf{t}) = \mathbf{K}_{j,0} + \sum_{i=1}^{n_\mathrm{e}}\left[ \mathbf{K}_{j,i}^{(1)} t_i + \mathbf{K}_{j,i}^{(3)} t_i^3 \right].
\end{equation}
Because the design variables can be bounded very similarly to Proposition \ref{prop:bounds} and scaled, all proven results hold true.

\begin{figure*}[t]
	\centering
	\begin{subfigure}{0.15\linewidth}
		\begin{tikzpicture}
		\centering
		\scaling{1.25}
		\point{a}{0.000000}{0.000000}
		\notation{1}{a}{\circled{$1$}}[below right=0mm]
		\point{b}{1.000000}{1.000000}
		\notation{1}{b}{\circled{$2$}}[below=1mm]
		\point{c}{0.000000}{2.000000}
		\notation{1}{c}{\circled{$3$}}[above right=0mm]
		\beam{2}{a}{b}
		\notation{4}{a}{b}[$1$]
		\beam{2}{b}{c}
		\notation{4}{b}{c}[$2$]
		\support{3}{a}[270]
		\support{3}{c}[270]
		\point{d1}{0.000000}{-0.500000}
		\point{d2}{1.000000}{-0.500000}
		\dimensioning{1}{d1}{d2}{-1.000000}[$1$]
		\point{e1}{1.250000}{0.000000}
		\point{e2}{1.250000}{1.000000}
		\dimensioning{2}{e1}{e2}{-0.7500000}[$1$]
		\point{e3}{1.250000}{2.000000}
		\dimensioning{2}{e2}{e3}{-0.75000000}[$1$]
		\load{1}{b}[0][-1.0][0.0]
		\notation{1}{b}{$1$}[above=10mm]
		\load{1}{b}[90][1.0][0.0]
		\notation{1}{b}{$1$}[left=9mm]
		\end{tikzpicture}
		\caption{}
	\end{subfigure}%
	\hfill\begin{subfigure}{0.1\linewidth}
		\vspace{9.6mm}
		\centering
		\begin{tikzpicture}[scale=0.75]
		\scaling{0.15}
		\point{a}{0}{0};
		\point{b}{5}{0};
		\point{c}{5}{10};
		\draw[black, fill=gray, fill opacity=0.2] (0.0,0.0) rectangle ++(1,2);
		\dimensioning{1}{a}{b}{-0.75}[$0.3$];
		\dimensioning{2}{b}{c}{1.75}[$h_i$];
		\end{tikzpicture}
		\vspace{11mm}
		\caption{}
	\end{subfigure}%
	\hfill\begin{subfigure}{0.17\linewidth}
		\vspace{4mm}\includegraphics[width=\linewidth]{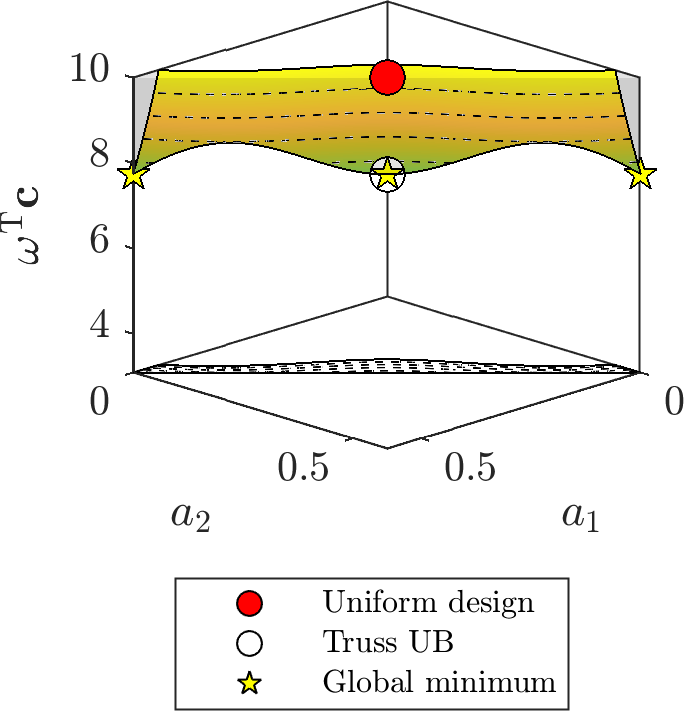}
		\vspace{4.5mm}
		\caption{}
	\end{subfigure}%
	\hfill\begin{subfigure}{0.17\linewidth}
		\vspace{4mm}\includegraphics[width=\linewidth]{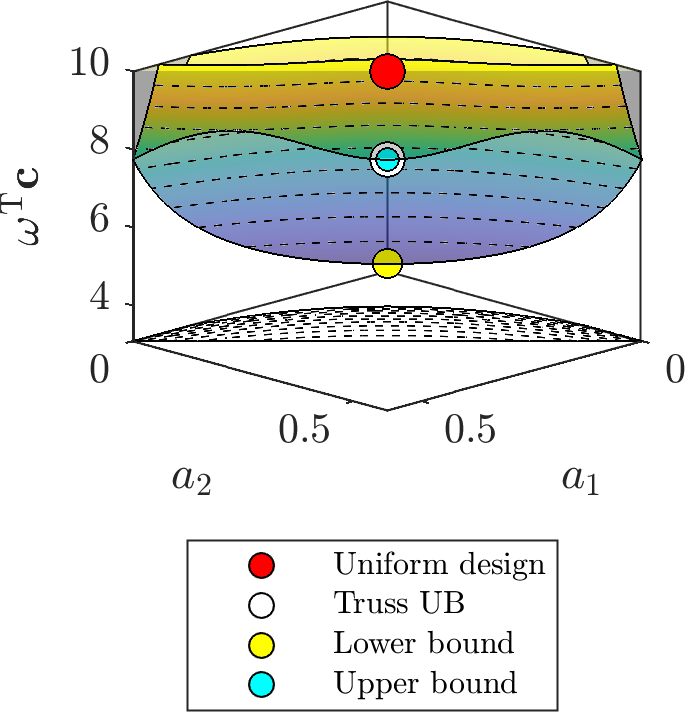}
		\vspace{4.5mm}
		\caption{}
	\end{subfigure}
	\hfill\begin{subfigure}{0.17\linewidth}
		\vspace{4mm}\includegraphics[width=\linewidth]{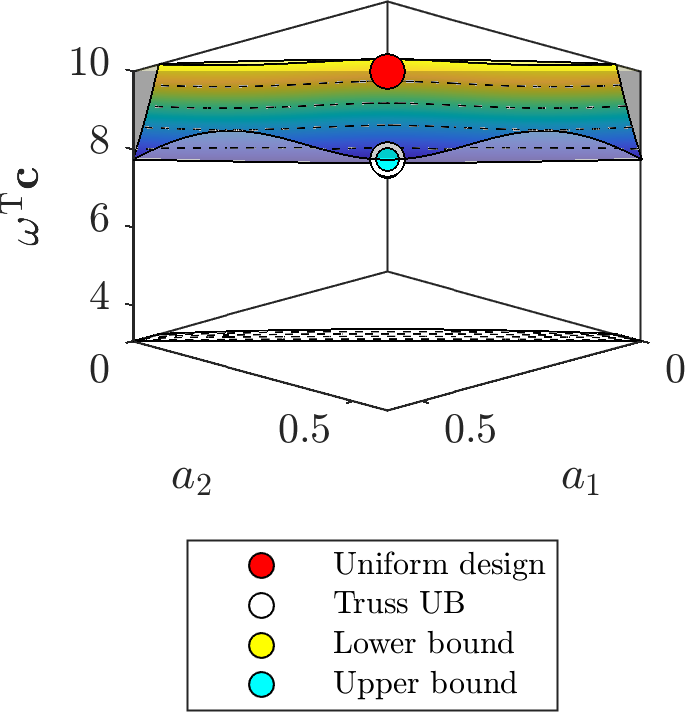}
		\vspace{4.5mm}
		\caption{}
	\end{subfigure}%
	\hfill\begin{subfigure}{0.17\linewidth}
		\vspace{4mm}\includegraphics[width=\linewidth]{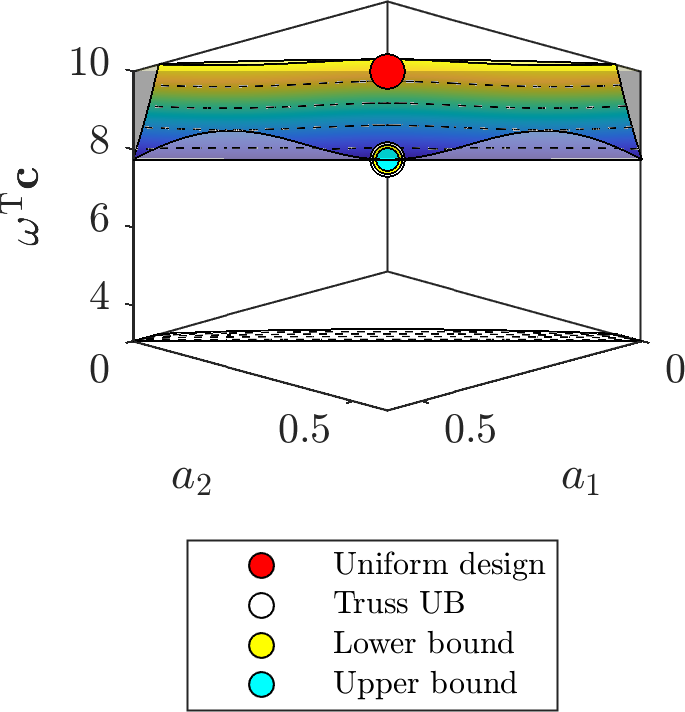}
		\vspace{4.5mm}
		\caption{}
	\end{subfigure}
	\caption{Frame structure composed of two elements: (a) boundary conditions, (b) the cross-section parametrization, and the sub-level set $\bm{\omega}^\mathrm{T} \mathbf{c}\le 10$ of the (c) feasible space and of the (d) second, (e) third, and (f) fourth outer approximations with the associated lower- and upper-bounds. Variables $a_1$ and $a_2$ stand for the cross-section areas of the two elements and $\bm{\omega}^\mathrm{T} \mathbf{c}$ denotes the corresponding weighted compliance of the two load cases (assuming the moments of inertia $I_i = 25/27 a_i^3$, $i \in \{1,2\}$), and $h_i$ is the cross-section height.}
	\label{fig:mg}
\end{figure*}
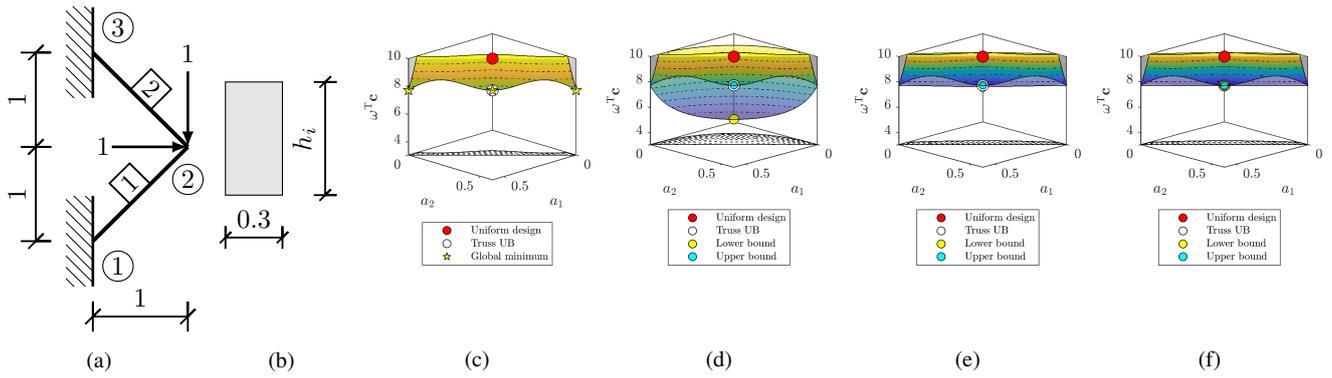

\section{Sample problems}\label{sec:examples}

This section investigates global topology optimization of selected small-scale structural design problems using the proposed strategy solved numerically by the \textsc{Mosek} optimizer \citep{mosek}. These examples demonstrate strengths and weaknesses of the presented approach: certificate of global $\varepsilon$-optimality using Theorem \ref{th:suff}, extraction\footnote{For rank computation we considered the eigenvalues with the absolute value smaller than $10^{-8}$ to be singular.} of all guaranteed globally optimal solutions based on the flat extension theorem \citep{Curto_1996}, but also higher computational demands when compared to selected local optimization techniques: OC and MMA adopting the nested approach, see, e.g., \citep{Bendsoe_2004}, \textsc{Matlab}'s inbuilt optimizer \texttt{fmincon} solving \eqref{eq:original} directly, and non-linear semidefinite programming (NSDP) formulation \eqref{eq:nsdp} solved by the \textsc{Penlab} optimizer \citep{Fiala2013}. Except for the nested approaches, all optimization problems were modeled using the \textsc{Yalmip} toolbox \citep{Lofberg2004}. Our implementation and the corresponding source codes written in \textsc{Matlab} can be accessed at \citep{tyburec_marek_2020_4048828}.

The first three examples involve two finite elements only to allow visualization of the feasible sets and provide intuition about the solution approach. In the later part, we investigate the influence of finite element types on optimal design and increase the number of elements to evaluate scalability of the approach. All computations were performed on a personal laptop with $16$~GB of RAM and Intel\textsuperscript{\textregistered} Core\texttrademark~i5-8350U CPU. Times of individual optimization approaches are measured to allow a simple comparison of the computational demands.

\subsection{Structure possessing multiple global optima}\label{sec:multiple}
As the first problem, we consider a frame structure composed of two Euler-Bernoulli frame elements, see Fig.~\ref{fig:mg}a. Two loads are applied, each of them acting as a separate load case, and weighted equally by $\bm{\omega} = \mathbf{1}$. Both these frame elements posses the Young modulus $E = 1$, and their overall volume is bounded by $\overline{V}=0.816597322$ from above\footnote{Fewer digits may prevent the solver from reaching all three global optima. Although an analytical formula for this specific $\overline{V}$ can be derived, we omit it for the sake of brevity.}. Accordingly with Fig.~\ref{fig:mg}b, the elements $i = \{1,2\}$ have rectangular cross-sections with areas $a_i = 0.3h_i$. Then, $I_i = \frac{1}{40}h^3$, which implies that $c_\mathrm{II} = 0$ and $c_\mathrm{III} = 25/27$ in Eq.~\eqref{eq:inertia}.

The feasible domain of the optimization problem shown in Fig.~\ref{fig:mg}c reveals that there are three global optima of the objective function value $7.738$, corresponding to the following cases: (i) $a_1^* = \overline{V}/\sqrt{2}$ and $a_2^* = 0$, (ii) $a_1^* = 0$ and $a_2^* = \overline{V}/\sqrt{2}$, and (iii) $a_1^* = a_2^* = \overline{V} \sqrt{2}/4$. All these solutions are extracted in the fourth relaxation of the moment-sum-of-squares hierarchy (Fig. \ref{fig:mg}f), which converged based on the rank condition \eqref{eq:rank} with the rank equal to $s=3$ and also based on Theorem~\ref{th:suff}, $\varepsilon = 6 \times 10^{-9}$. Notice that in all the relaxations, the upper bounds recovered by Theorem~\ref{th:feasible} are global minima, Figs.~\ref{fig:mg}d-\ref{fig:mg}f.

Because all local minima are also global, all tested optimization algorithms converge to the optimal objective function value, see Table \ref{tab:multiple2}. Among these algorithms, OC and MMA exhibited the best performance in terms of computational time.

\begin{table}[!b]
	\centering\setlength{\tabcolsep}{5pt}
	\begin{tabular}{lccccc}
		method & $a_1$ & $a_2$ & LB & UB & time [s]\\ 
		\hline
		OC &  $0.289$ & $0.289$ & - & $7.738$ & $0.009$ \\
		MMA & $0.289$ & $0.289$ & - & $7.738$ & $0.011$ \\
		\texttt{fmincon} & $0.289$ & $0.289$ & - & $7.738$ & $0.113$\\
		NSDP & $0.289$ & $0.289$ & - & $7.738$ & $0.409$\\
		PO$^{(2)}$, Th.~\ref{th:suff} & $0.289$ & $0.289$ & $5.065$ & $7.738$ & $0.023$\\
		PO$^{(3)}$, Th.~\ref{th:suff} & $0.289$ & $0.289$ & $7.647$ & $7.738$ & $0.071$\\
		PO$^{(4)}$, Th.~\ref{th:suff} & $0.289$ & $0.289$ & $7.738$ & $7.738$ & $0.438$
	\end{tabular}
\caption{Different optimization methods applied to the first optimization problem. LB denotes lower bound, UB abbreviates feasible upper bounds, and PO stands for polynomial optimization. The entries $a_i$ denote cross-section areas of the $i$-th element, or the areas constructed from the first-order moments in the case of PO.}
\label{tab:multiple2}
\end{table}

\subsection{Irreducible positive optimality gap}\label{sec:failed}

Let us now modify the optimization problem described in the preceding section by fixing the volume bound to some $\overline{V} \in (0.816597322, 2.73603242)$. Whilst the boundary points of this open interval match the cases when three global optima occur, the interval interior removes $a_1 = a_2 = \overline{V} \sqrt{2}/4$ from the set of globally optimal solutions. In what follows, we set $\overline{V}$ to the center of the interval.

\begin{figure}[!b]
	\centering
	\begin{subfigure}{0.33\linewidth}
		\includegraphics[width=\linewidth]{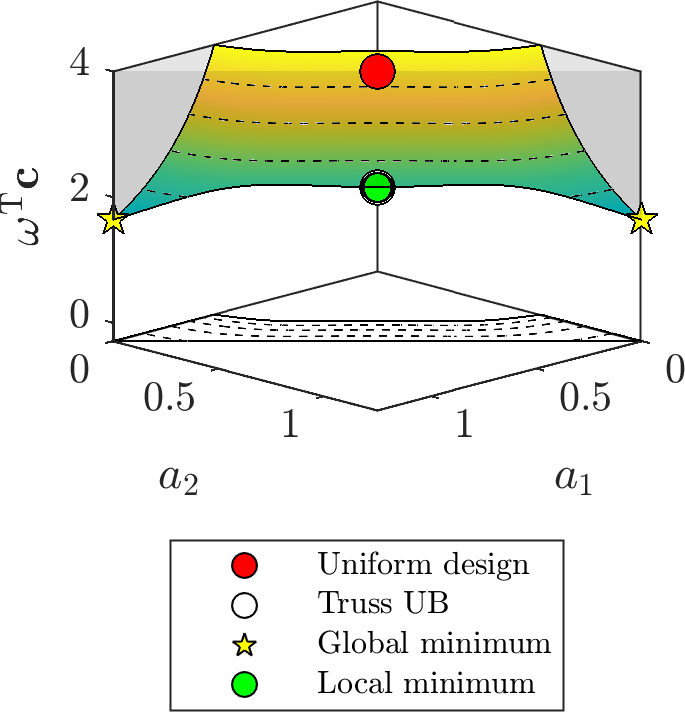}
		\caption{}
	\end{subfigure}%
	\hfill\begin{subfigure}{0.33\linewidth}
		\includegraphics[width=\linewidth]{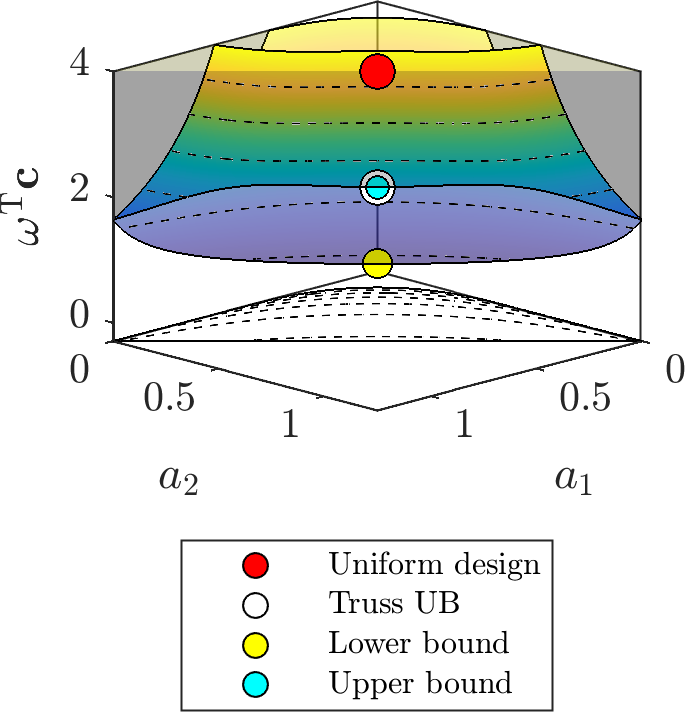}
		\caption{}
	\end{subfigure}%
	\hfill\begin{subfigure}{0.33\linewidth}
		\includegraphics[width=\linewidth]{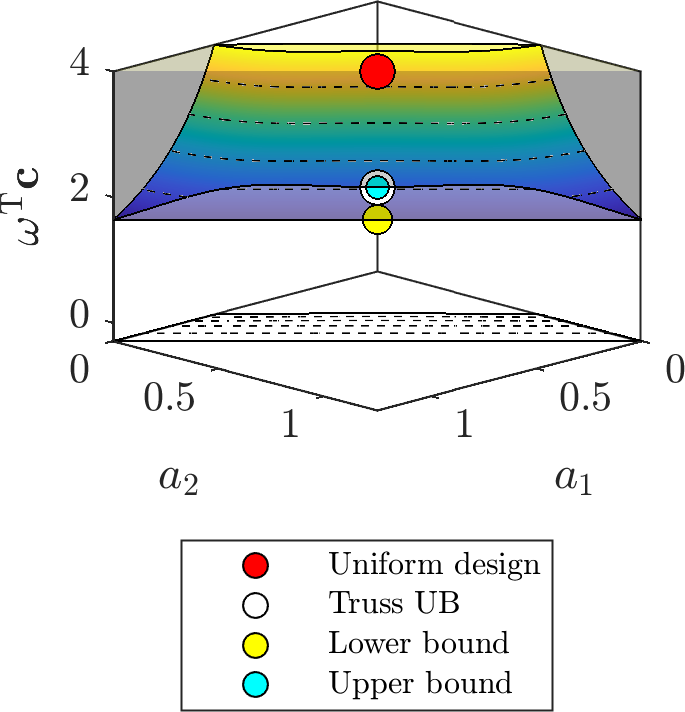}
		\caption{}
	\end{subfigure}
	\caption{Frame structure possessing a~non-zero optimality gap. The sub-level set $\bm{\omega}^\mathrm{T} \mathbf{c}\le 4$ of the (a) feasible space and of the (b) second, and (c) third outer approximations with the associated lower- and upper-bounds. Variables $a_1$ and $a_2$ stand for the cross-section areas of the two elements and $\bm{\omega}^\mathrm{T} \mathbf{c}$ denotes the corresponding weighted compliance of the two load cases (assuming the moments of inertia $I_i = 25/27 a_i^3$, $i \in \{1,2\}$).}
	\label{fig:sufficient}
\end{figure}

\begin{table}[!t]
	\centering\setlength{\tabcolsep}{5pt}
	\begin{tabular}{lccccc}
		method & $a_1$ & $a_2$ & LB & UB & time [s]\\ 
		\hline
		OC &  $0.628$ & $0.628$ & - & $2.161$ & $0.003$ \\
		MMA & $0.628$ & $0.628$ & - & $2.161$ & $0.004$ \\
		\texttt{fmincon} & $0.628$ & $0.628$ & - & $2.161$ & $0.049$\\
		NSDP & $0.628$ & $0.628$ & - & $2.161$ & $0.200$\\
		PO$^{(2)}$, Th.~\ref{th:suff} & $0.628$ & $0.628$ & $0.936$ & $2.161$ & $0.022$\\
		PO$^{(3)}$, Th.~\ref{th:suff} & $0.628$ & $0.628$ & $1.640$ & $2.161$ & $0.063$\\
		\multirow{2}{*}{PO$^{(3)}$, Eq.~\eqref{eq:rank}} & $1.256$ & $0.000$ & $1.640$ & $1.640$ & $0.063$\\
		& $0.000$ & $1.256$ & $1.640$ & $1.640$ & $0.063$
	\end{tabular}
	\caption{Different optimization methods applied to the second optimization problem. LB denotes lower bound, UB abbreviates feasible upper bounds, and PO stands for polynomial optimization. The entries $a_i$ denote cross-sectional areas of the $i$-th element, or the areas constructed from the first-order moments in the case of PO.}
	\label{tab:sufficient}
\end{table}

\begin{figure*}[b]
	\centering
	\begin{subfigure}{0.15\linewidth}
		\begin{tikzpicture}
		\centering
		\scaling{1.25}
		\point{a}{0.000000}{0.000000}
		\notation{1}{a}{\circled{$1$}}[below right=0mm]
		\point{b}{1.000000}{1.000000}
		\notation{1}{b}{\circled{$2$}}[below=1mm]
		\point{c}{0.000000}{2.000000}
		\notation{1}{c}{\circled{$3$}}[above right=0mm]
		\beam{2}{a}{b}
		\notation{4}{a}{b}[$1$]
		\beam{2}{b}{c}
		\notation{4}{b}{c}[$2$]
		\support{3}{a}[270]
		\support{3}{c}[270]
		\point{d1}{0.000000}{-0.500000}
		\point{d2}{1.000000}{-0.500000}
		\dimensioning{1}{d1}{d2}{-1.000000}[$1$]
		\point{e1}{1.250000}{0.000000}
		\point{e2}{1.250000}{1.000000}
		\dimensioning{2}{e1}{e2}{-0.7500000}[$1$]
		\point{e3}{1.250000}{2.000000}
		\dimensioning{2}{e2}{e3}{-0.75000000}[$1$]
		\load{1}{b}[0][-1.0][0.0]
		\notation{1}{b}{$1$}[above=10mm]
		\load{1}{b}[90][1.0][0.0]
		\notation{1}{b}{$1$}[left=9mm]
		\end{tikzpicture}
		\caption{}
	\end{subfigure}%
	\hfill\begin{subfigure}{0.1\linewidth}
		\vspace{4mm}
		\centering
		\begin{tikzpicture}[scale=1]
		\scaling{0.2}
		\point{a}{0}{0};
		\point{a1}{0}{1};
		\point{a2}{0}{9};
		\point{a3}{0}{10};
		\point{b}{5}{0};
		\point{b1}{5}{1};
		\point{b2}{5}{9};
		\point{b3}{5}{10};
		\point{e1}{2}{1};
		\point{e2}{2}{9};
		\point{f1}{3}{1};
		\point{f2}{3}{9};
		\point{c}{5}{10};
		\point{A1}{1}{0.5};
		\point{A2}{2.5}{2.5};
		\point{A}{1}{4};
		\point{B}{-1}{4};
		\draw (B) -- node[above]{$t_{\mathrm{p},i}$} (A);
		\draw[black, fill=gray, fill opacity=0.2] (a)--(b)--(b1)--(f1)--(f2)--(b2)--(b3)--(a3)--(a2)--(e2)--(e1)--(a1)--(a);
		\draw [{Latex}-](A1) -- (A);
		\draw [{Latex}-](A2) -- (A);
		\dimensioning{1}{a}{b}{-0.75}[$5 t_{\mathrm{p},i}$];
		\dimensioning{2}{b}{c}{1.5}[$10 t_{\mathrm{p},i}$];
		\end{tikzpicture}
		\vspace{4.6mm}
		\caption{}
	\end{subfigure}%
	\hfill\begin{subfigure}{0.17\linewidth}
		\vspace{4mm}\includegraphics[width=\linewidth]{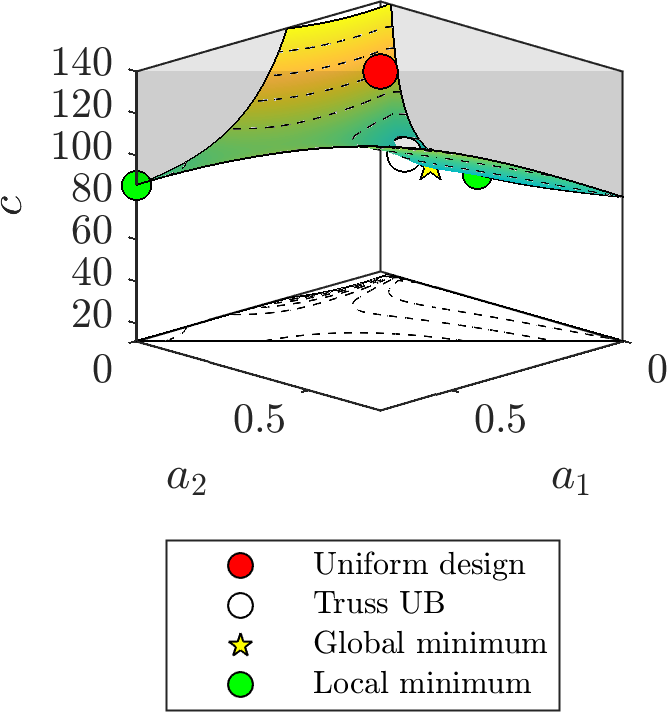}
		\vspace{4mm}
		\caption{}
	\end{subfigure}%
	\hfill\begin{subfigure}{0.17\linewidth}
		\vspace{4mm}\includegraphics[width=\linewidth]{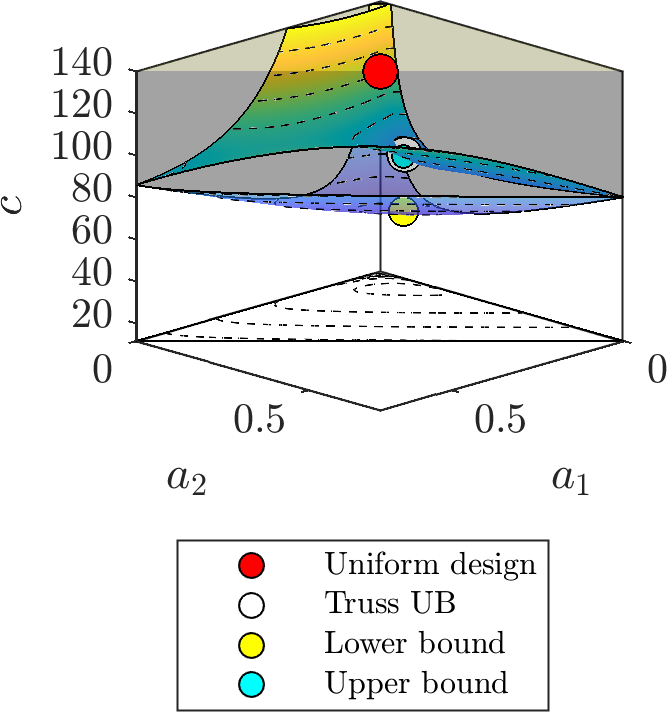}
		\vspace{4mm}
		\caption{}
	\end{subfigure}
	\hfill\begin{subfigure}{0.17\linewidth}
		\vspace{4mm}\includegraphics[width=\linewidth]{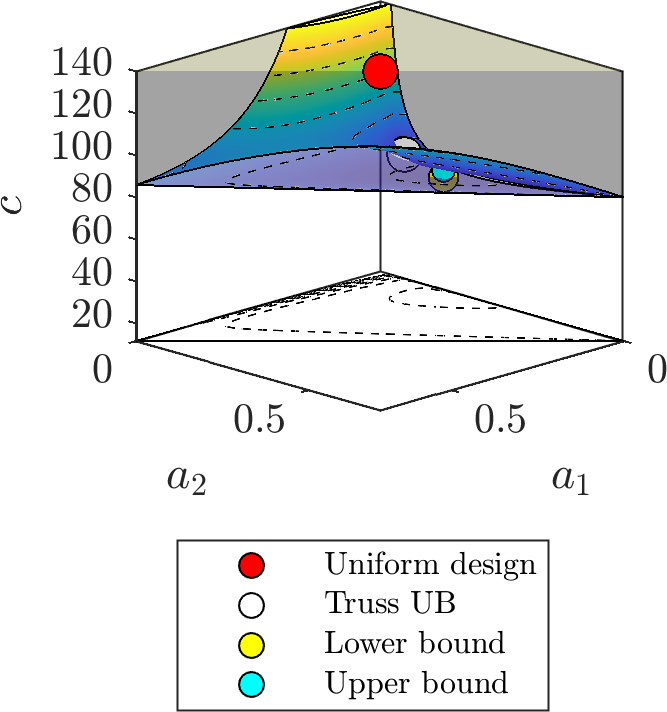}
		\vspace{4mm}
		\caption{}
	\end{subfigure}%
	\hfill\begin{subfigure}{0.17\linewidth}
		\vspace{4mm}\includegraphics[width=\linewidth]{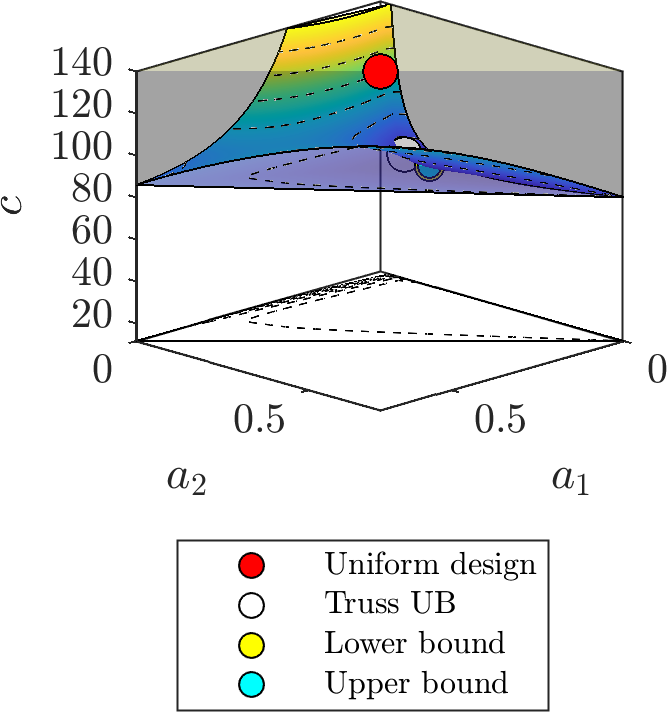}
		\vspace{4mm}
		\caption{}
	\end{subfigure}
	\caption{Frame structure with self-weight: (a) boundary conditions, (b) the cross-section parametrization, and the sub-level set $c\le 140$ of the (c) feasible space and of the (d) first, (e) second, and (f) third outer approximations with the associated lower- and upper-bounds. Variables $a_1$ and $a_2$ stand for the cross-section areas of the two elements and $c$ denotes the corresponding compliance (assuming the moments of inertia $I_i = 41/54 a_i^2$, $i \in \{1,2\}$), and $t_{\mathrm{p},i}$ stands for the flanges and web thickness.}
	\label{fig:sw}
\end{figure*}
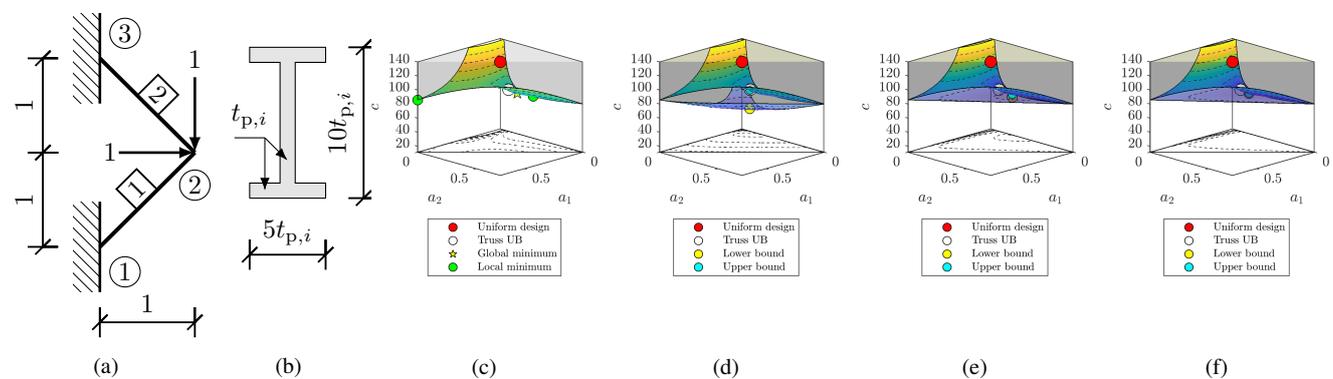

Solving this modified optimization problem with the moment-sum-of-squares hierarchy produces the sequence of lower- and upper-bounds shown in Fig.~\ref{fig:sufficient}. Although the hierarchy exhibited a finite convergence based on the rank condition \eqref{eq:rank} with $s=2$, the corresponding optimality gap remains strictly positive ($\varepsilon = 0.521$) and cannot be reduced in the subsequent relaxations. Clearly, all outer convex approximations must contain the convex combination of their limit points. Hence, if the limit points denote the global optima, also their convex combinations attain the globally optimal objective function value. Therefore, they are also optimal for the associated relaxation, but may lack feasibility to the original problem. Depending on the optimization algorithm and its settings, one can either reach a lower bound that is actually feasible for the original problem (as was the case in Section \ref{sec:multiple}), i.e., a zero optimality gap, or a~positive optimality gap that cannot be further reduced, which is the case here.

For this particular problem, all local optimization techniques, using their default starting points and settings, missed the global optima, see Table~\ref{tab:sufficient}. In fact, they approached the feasible upper-bound that was provided by Theorem \ref{th:feasible}.

\subsection{Frame structure with self-weight}

For the previous examples, $\mathbf{f}(\mathbf{a}) = \mathbf{f}$ was constant, so that the optimum designs utilized the entire available volume $\overline{V}$. In these cases, the volume inequality constraint could have been changed to equality, and, therefore, one design variable eliminated. However, such a~procedure cannot be applied when design-dependent loads are present.

\begin{table}[t]
	\centering\setlength{\tabcolsep}{5pt}
	\begin{tabular}{lccccc}
		method & $a_1$ & $a_2$ & LB & UB & time [s]\\ 
		\hline
		OC &  $0.022$ & $0.166$ & - & $70.442$ & $1.129$\\
		MMA &  $0.022$ & $0.166$ & - & $70.442$ & $0.935$\\
		NSDP & $0.707$ & $0.000$ & - & $85.846$ & $1.448$\\
		PO$^{(1)}$, Th.~\ref{th:suff} & $0.050$ & $0.119$ & $48.246$ & $74.171$ & $0.006$\\
		PO$^{(2)}$, Th.~\ref{th:suff} & $0.034$ & $0.220$ & $68.328$ & $71.594$ & $0.015$\\
		PO$^{(3)}$, Th.~\ref{th:suff} & $0.022$ & $0.166$ & $70.442$ & $70.442$ & $0.058$
	\end{tabular}
	\caption{Different optimization methods applied to the optimization problem with self-weight. LB denotes lower bound, UB abbreviates feasible upper bounds, and PO stands for polynomial optimization. The entries $a_i$ denote cross-sectional areas of the $i$-th element, or the areas constructed from the first-order moments in the case of PO.}
	\label{tab:sw}
\end{table}

To visualize this, let our third illustration be the single-load-case frame structure in Fig.~\ref{fig:sw}a, composed of two frame elements with $E = 1$ with I-shaped cross-sections, Fig.~\ref{fig:sw}b, parametrized by the thickness $t_{\mathrm{p},i}$. The overall volume is bounded from above by $\overline{V} = 1$. The self-weight applies in the vertical direction and is parametrized by the material density $\rho = 10$. For the considered cross-sections, we have $a_i = 18 t_{\mathrm{p},i}^2$ and $I_i = 246 t_{\mathrm{p},i}^4$. Hence, $c_{\mathrm{II}} = 41/54$ and $c_\mathrm{III} = 0$ in Eq.~\eqref{eq:inertia}.

The feasible domain of this optimization problem, Fig.~\ref{fig:sw}c, reveals three local optima, and one of them is the global solution. Computation of the optimum by the moment-sum-of-squares hierarchy required three relaxations, Figs.~\ref{fig:sw}d--\ref{fig:sw}f, which converged based on both the rank condition \eqref{eq:rank} with $s=1$ and on Theorem~\ref{th:suff} with $\varepsilon = -7 \times 10^{-8}$; the slightly negative value of $\varepsilon$ is due to the numerical accuracy of the optimizer. Also notice that the upper-bounds based on Theorem \ref{th:feasible} are of very high qualities, see Table~\ref{tab:sw}.

Using local optimization techniques, only OC and MMA were able to arrive at the global optimum, see Table~\ref{tab:sw}. Among other formulations, NSDP approached the worst local optimum and  \texttt{fmincon} failed even to find a feasible solution.

\subsection{Different element types on a cantilever beam}\label{sec:examples_cantilever}

\begin{figure}[t]
	\newcommand{\snewpoint}[4]{\dpoint{#1m}{#2}{0}{0};\dpoint{#1a}{#2}{-#3}{-#4};\dpoint{#1b}{#2}{-#3}{#4};\dpoint{#1c}{#2}{#3}{#4};\dpoint{#1d}{#2}{#3}{-#4};\dpoint{#1am}{#2}{0}{-#4};\dpoint{#1bm}{#2}{0}{#4};}
	\newcommand{\surfcube}[2]{\dbeam{3}{#1a}{#1b};\dbeam{2}{#1b}{#1c};\dbeam{2}{#1c}{#1d};\dbeam{3}{#1d}{#1a};\dbeam{3}{#2a}{#2b};\dbeam{2}{#2b}{#2c};\dbeam{2}{#2c}{#2d};\dbeam{3}{#2d}{#2a};\dbeam{3}{#1a}{#2a};\dbeam{2}{#1b}{#2b};\dbeam{2}{#1c}{#2c};\dbeam{2}{#1d}{#2d};}
	\newcommand{\surflastcube}[2]{\dbeam{3}{#1a}{#1b};\dbeam{2}{#1b}{#1c};\dbeam{2}{#1c}{#1d};\dbeam{3}{#1d}{#1a};\dbeam{2}{#2a}{#2b};\dbeam{2}{#2b}{#2c};\dbeam{2}{#2c}{#2d};\dbeam{2}{#2d}{#2a};\dbeam{3}{#1a}{#2a};\dbeam{2}{#1b}{#2b};\dbeam{2}{#1c}{#2c};\dbeam{2}{#1d}{#2d};}
	\begin{tikzpicture}
	%\showpoint
	\snewpoint{a}{0.0}{0.2}{1.0}; %name, x, dy/2, dz/2
	\snewpoint{b}{0.9}{0.2}{1.0};
	\snewpoint{c}{1.8}{0.2}{1.0};
	\snewpoint{d}{2.7}{0.2}{1.0};
	\snewpoint{e}{3.6}{0.2}{1.0};
	\snewpoint{f}{4.5}{0.2}{1.0};
	\dpoint{ax}{-1.2}{1}{0};
	\dpoint{bx}{-1}{1}{0};
	\support{3}{aam}[270];
	\support{3}{abm}[270];
	\support{3}{am}[270];
	\surfcube{a}{b};
	\surfcube{b}{c};
	\surfcube{c}{d};
	\surfcube{d}{e};
	\surflastcube{e}{f};
	%\dsupport{2}{am}[yz];
	\dlineload{1}{yz}{fbm}{fam}[.75][.75];
	\dlineload{1}{xz}{fbm}{fam}[1.3][1.3];
	\ddimensioning{xz}[-1.2]{aa}{ba}{.5}[$1$][1.2];
	\ddimensioning{xz}[-1.2]{ba}{ca}{.5}[$1$][1.2];
	\ddimensioning{xz}[-1.2]{ca}{da}{.5}[$1$][1.2];
	\ddimensioning{xz}[-1.2]{da}{ea}{.5}[$1$][1.2];
	\ddimensioning{xz}[-1.2]{ea}{fa}{.5}[$1$][1.2];
	\ddimensioning{xz}[-1.75]{aa}{fa}{.5}[$5$][1.75];
	\ddimensioning{zx}[1.5]{aa}{ab}{.5}[$1$][1.5];
	\dnotation{1}{fm}{$\cos(30^\circ)$}[below right=4mm and -2mm];
	\dnotation{1}{fm}{$\sin(30^\circ)$}[above right=7mm and 4mm];
	\dnotation{4}{am}{bm}[$1$];
	\dnotation{4}{bm}{cm}[$2$];
	\dnotation{4}{cm}{dm}[$3$];
	\dnotation{4}{dm}{em}[$4$];
	\dnotation{4}{em}{fm}[$5$];
	\dscaling{3}{0.5}
	\setaxis{3}[$A$][$B$][$C$][$x$][$y$][$z$]
	\setaxis{4}[right][below left][below];
	\daxis{3}{0}[ax][bx][0][200.0][180];
	\end{tikzpicture}
	\caption{Boundary conditions of the cantilever beam design problem.}
	\label{fig:cantilever}
\end{figure}
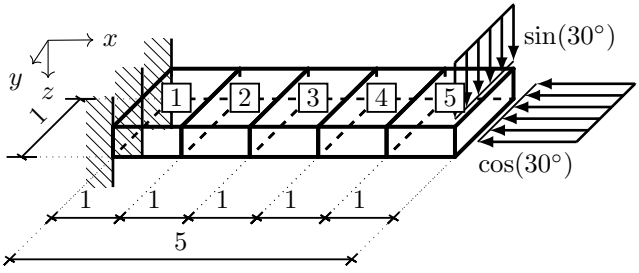

\begin{figure}[b]
	\includegraphics[width=\linewidth]{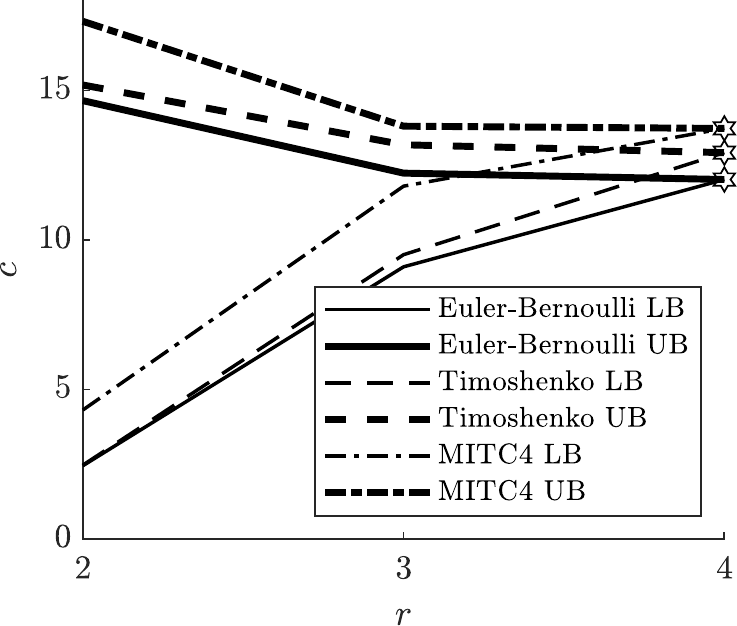}
	\caption{Convergence of the moment-sum-of-squares hierarchy for the cantilever problem with three finite element types. Variable $c$ denotes compliance, $r$ stands for the relation degree, and LB with UB abbreviate lower- and upper-bound.}
	\label{fig:cantilever_convergence}
\end{figure}

A certain generality of the developed approach is illustrated on a~cantilever beam/plate design problem, Fig.~\ref{fig:cantilever}. The dimensions of the cantilever are $5$ in length and $1$ in width, and the thicknesses of $5$ finite elements are to be found in the optimization. The beam is made of a linear-elastic material with the Young modulus $E=1$ and Poisson ratio $\nu=0.25$. This structure is subjected to a tip distributed load of magnitude $1$ induced under $30^\circ$ angle with respect to the midline/midsurface. We optimize the frame/shell thicknesses $t_i$ (of rectangular cross-sections) while satisfying $\overline{V}=10$. The shear correction factor is set to $5/6$ where appropriate.

In what follows, we compare the optimization results of the cantilever problem for three finite element types: Euler-Bernoulli and Timoshenko frame elements, and the quadrilateral Mixed Interpolation Tensorial Component (MITC4) shell element~\citep{Bathe_1986}. For both of the frame element types, we have $c_\mathrm{II} = 0$ and $c_\mathrm{III} = I_i (a)/a_i^3 = 1/12$ in Eq~\eqref{eq:inertia}, whereas $c_\mathrm{II} = 0$ and $c_\mathrm{III} = 1$ for the MITC4 element.

\begin{table}[t]
	\centering\setlength{\tabcolsep}{5pt}
	\begin{tabular}{cccc}
		& Euler-Bernoulli & Timoshenko & MITC4\\
		\hline
		$a_1^*$ & $2.775$ & $2.724$ & $2.754$\\ 
		$a_2^*$ & $2.454$ & $2.414$ & $2.462$\\
		$a_3^*$ & $2.086$ & $2.060$ & $2.091$\\
		$a_4^*$ & $1.639$ & $1.643$ & $1.651$\\
		$a_5^*$ & $1.047$ & $1.159$ & $1.041$\\
		$c^*$ & $12.025$ & $12.922$ & $13.734$\\
		time [s] & $65.363$ & $58.633$ & $967.086$\\
		Th. \ref{th:suff}, $\varepsilon$ & $-2\times 10^{-9}$ & $-9 \times 10^{-10}$ & $-2 \times 10^{-9}$\\
		Eq.~\ref{eq:rank}, $s$ & $1$ & $1$ & $1$
	\end{tabular}
	\caption{Globally optimal thicknesses $a_1^*,\dots,a_5^*$ and compliances $c^*$ for the cantilever problem for three element types: Euler-Bernoulli and Timoshenko frame elements, and the MITC4 shell element. Variables $\varepsilon$ and $s$ denote the optimality gap in Theorem \ref{th:suff} and the rank of the moment matrices according to \eqref{eq:rank}, respectively.}
	\label{tab:optimal solutions}
\end{table}

\begin{figure*}[t]
	\begin{subfigure}{0.375\linewidth}
		\centering
		\begin{tikzpicture}
		\scaling{2.3};
		
		\point{a}{0}{0};
		\point{b}{1}{0};
		\point{c}{2}{0};
		\point{d}{0}{1};
		\point{e}{1}{1};
		\point{f}{2}{1};
		\point{g}{0}{2};
		\point{h}{1}{2};
		\point{i}{2}{2};
		
		%[1 2; 1 5; 1 6; 1 8; 2 3; 2 4; 2 5; 2 6; 2 7; 2 9; 3 4; 3 5; 3 6; 3 7; 3 8; 4 5; 4 8; 4 9; 5 6; 5 7; 5 8; 5 9; 6 7; 6 8; 6 9; 7 8; 8 9]
		
		\beam{2}{a}{b};
		\beam{2}{a}{e};
		\beam{2}{a}{f};
		\beam{2}{a}{h};
		\beam{2}{b}{c};
		\beam{2}{b}{d};
		\beam{2}{b}{e};
		\beam{2}{b}{f};
		\beam{2}{b}{g};
		\beam{2}{c}{d};
		\beam{2}{c}{e};
		\beam{2}{c}{f};
		\beam{2}{c}{h};
		\beam{2}{d}{e};
		\beam{2}{e}{f};
		\beam{2}{d}{h};
		\beam{2}{e}{g};
		\beam{2}{e}{h};
		\beam{2}{g}{h};
		\beam{2}{f}{g};
		\beam{2}{f}{h};
		
		\support{3}{a}[270];
		\support{3}{d}[270];
		\support{3}{g}[270];
		
		\load{1}{c}[90][0.7][-0.7];
		\load{1}{f}[90][1][0.12];
		\notation{1}{c}{$2$}[below right=2mm];
		\notation{1}{f}{$3.5$}[above right=2mm];
		
		\draw[->] (0.0,0)--(1,0);
		\node at (0.6, 0.15) {$x$};
		\draw[->] (0.0,0)--(0,1);
		\node at (0.15, 0.6) {$y$};
		
		\dimensioning{1}{a}{b}{-1.2}[$1$]
		\dimensioning{1}{b}{c}{-1.2}[$1$]
		\dimensioning{2}{a}{d}{-0.8}[$1$]
		\dimensioning{2}{d}{g}{-0.8}[$1$]
		
		\notation{1}{a}{\circled{$1$}}[align=center];
		\notation{1}{b}{\circled{$2$}}[align=center];
		\notation{1}{c}{\circled{$3$}}[align=center];
		\notation{1}{d}{\circled{$4$}}[align=center];
		\notation{1}{e}{\circled{$5$}}[align=center];
		\notation{1}{f}{\circled{$6$}}[align=center];
		\notation{1}{g}{\circled{$7$}}[align=center];
		\notation{1}{h}{\circled{$8$}}[align=center];
		\end{tikzpicture}
		%\vspace{8mm}
		\caption{}
	\end{subfigure}%
	\hfill\begin{subfigure}{0.15\linewidth}
		\centering
		\begin{tikzpicture}[scale=1]
		\scaling{0.2}
		\point{a}{0}{0};
		\point{a1}{4}{0};
		\point{a2}{5}{0};
		\draw[black, fill=gray, fill opacity=0.2] (0,0) circle (1);
		\draw[black, fill=white, fill opacity=1.0] (0,0) circle (0.8);
		%		\draw (B) -- node[above]{$t_\mathrm{p}$} (A);
		%		\draw[black, fill=gray, fill opacity=0.2] (a)--(b)--(b1)--(f1)--(f2)--(b2)--(b3)--(a3)--(a2)--(e2)--(e1)--(a1)--(a);
		%		\draw [{Latex}-](A1) -- (A);
		%		\draw [{Latex}-](A2) -- (A);
		\dimensioning{1}{a}{a1}{1.2}[\raisebox{2mm}{$4r_i$}];
		\dimensioning{1}{a1}{a2}{1.2}[\raisebox{2mm}{$r_i$}];
		\end{tikzpicture}
		\caption{}
	\end{subfigure}%
	\hfill\begin{minipage}{0.43\linewidth}
		\begin{subfigure}{0.4\linewidth}
			\centering
			\includegraphics[width=0.88\linewidth]{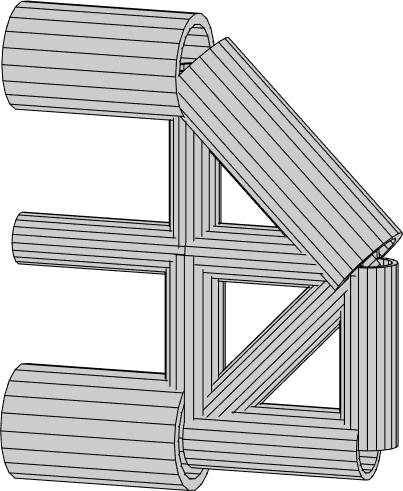}
			\caption{}
		\end{subfigure}%
		\hfill\begin{subfigure}{0.4\linewidth}
			\centering
			\includegraphics[width=0.88\linewidth]{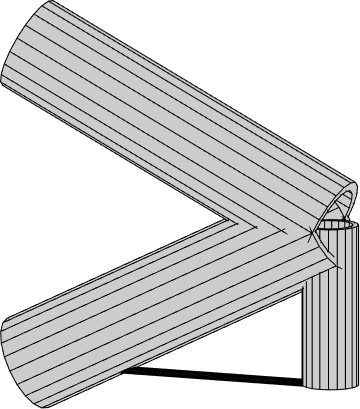}
			\caption{}
		\end{subfigure}\\
		\begin{subfigure}{0.4\linewidth}
			\centering
			\includegraphics[width=0.88\linewidth]{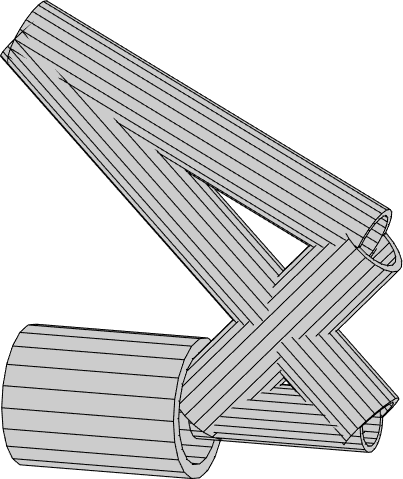}
			\caption{}
		\end{subfigure}%
		\hfill\begin{subfigure}{0.4\linewidth}
			\centering
			\includegraphics[width=0.88\linewidth]{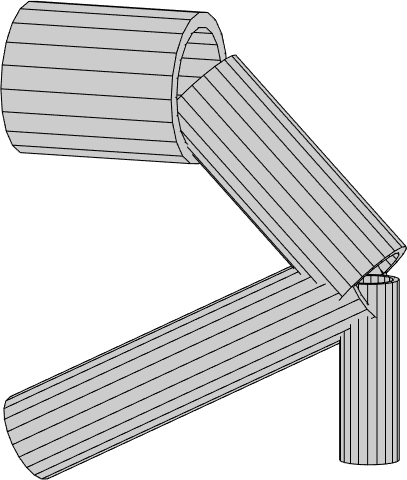}
			\caption{}
		\end{subfigure}
	\end{minipage}
	\caption{(a) Ground structure of the $22$-elements frame optimization problem, (b) cross-section parameterized by $r_i$, and optimized designs of compliances: (c) $c=3276.3$ obtained by PO$^{(1)}$, (d) $c^* = 1668.6$ resulting from PO$^{(2)}$ and OC, (e) $c = 1697.7$ reached by MMA and \texttt{fmincon}, and (f) $c = 1741.1$ optimized by NSDP.}
	\label{fig:10}
\end{figure*}
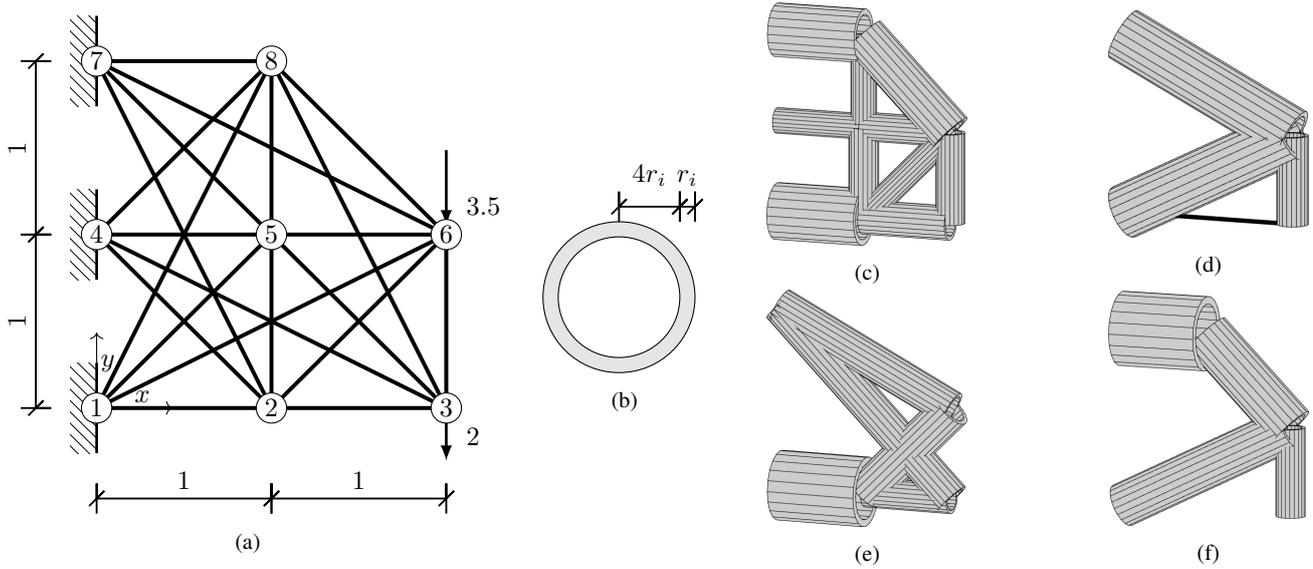

The moment-sum-of-squares hierarchy required three steps (degree-four relaxation) to converge in all three cases, Fig.~\ref{fig:cantilever_convergence}, and approached very similar optimal thicknesses, Table~\ref{tab:optimal solutions}. As expected, the lowest compliance is provided by the Euler-Bernoulli frame elements, which neglect the shear effects. We account for these effects in the Timoshenko frame elements, increasing so the value of optimal compliance. Another generalization occurs when using the MITC4 shell elements, which not only consider the effects of shear, but further incorporate effects induced by bending about $z$ axis, recall Fig.~\ref{fig:cantilever}. Therefore, the optimal compliance associated with the MITC4 elements is the highest. These results thus reveal the importance of using an appropriate finite element type for particular problem, as neglecting a~physical phenomenon may result in a~suboptimal design. Moreover, this influence on optimal minimum-energy designs can be rigorously studied by the proposed approach.

\subsection{22-elements frame structure}

Our final example investigates topology optimization of a~$22$-element frame structure shown in Fig.~\ref{fig:10}a. Two loads are applied at nodes $3$ and $6$ in a single load case. In addition, we set $E=1$ and $\overline{V} = 0.5$. All the structural elements possess a thin-walled circular cross-section with the radius $5 r_i$ and the wall thickness $r_i$, Fig.~\ref{fig:10}b. Hence, $a_i = 9 \pi r_i^2$ and $I_i = 46.125 \pi r_i^4$, so that $c_\mathrm{II} = 46.125/(81\pi)$ and $c_\mathrm{III} = 0$.

\begin{table}[b]
	\centering
	\centering\setlength{\tabcolsep}{3pt}
	\begin{tabular}{lccc}
		method & LB & UB & time [s]\\ 
		\hline
		OC & - & $1668.585$ & $2.454$\\
		MMA & - & $1697.749$ & $13.816$\\
		\texttt{fmincon} & - & $1697.665$ & $0.650$ \\
		NSDP & - & $1741.062$ & $4.406$\\
		PO$^{(1)}$, Th.~\ref{th:suff}, Eq.~\eqref{eq:rank} & $1062.105$ & $3276.294$ & $0.103$\\
		PO$^{(2)}$, Th.~\ref{th:suff}, Eq.~\eqref{eq:rank} & $1668.584$ & $1668.584$ & $1492.842$
	\end{tabular}
	\caption{Different optimization methods applied to the $22$-frame structure design problem. LB denotes lower bounds, UB abbreviates feasible upper bounds, and PO stands for polynomial optimization.}
	\label{tab:22}
\end{table}

The moment-sum-of-squares hierarchy requires two relaxations to achieve a guaranteed global optimum, both based on Theorem \ref{th:suff} with $\varepsilon = 2\times 10^{-5}$ and on Eq.~\eqref{eq:rank} with $s = 1$. However, even the second relaxation is fairly computationally expensive (see Table \ref{tab:22}), prohibiting solution of higher relaxations of similarly-sized problems on standard hardware.

Evaluation of the local optimization algorithms revealed that only OC converged to the global optimum ($c^* = 1668.6$ shown in Fig.~\ref{fig:10}d). The remaining optimization approaches reached local optima of comparable performance but considerably different topologies: MMA and \texttt{fmincon} converged to the design shown in Fig.~\ref{fig:10}e with $c=1697.7$, and NSDP reached the design in Fig.~\ref{fig:10}f with $c=1741.1$.

\section{Conclusions}\label{sec:conclusion}

Our contribution has addressed the fundamental question in the structural design: how to find globally-optimal minimum-compliant bending-resistant structures in discrete topology optimization with continuous design variables. For the cases of frame and shell structures, multiple loading conditions and design-dependent loads, we have formulated this optimization problem as a~(non-linear) semidefinite program constrained by a polynomial matrix inequality. The feasible space of this optimization problem forms a~semialgebraic set; hence, powerful results on polynomial optimization---the moment-sum-of-squares hierarchy---facilitate computation of the global solutions.

This hierarchy generates a sequence of tightening outer convex approximations in the space of moments of the probability measures, so that the first-order moments converge monotonically to the convex hull of the original problem. Therefore, a~non-decreasing sequence of lower bounds is established. Using the first-order moments only, we have shown that a sequence of feasible upper bounds can be obtained by a simple correction. Consequently, because lower and upper bounds are available in each relaxation, the upper-bound design quality can be assessed, establishing a~sufficient condition of global $\varepsilon$-optimality. This condition is very simple to check and complements the traditional rank-based certificate of global optimality, e.g., \citep{Henrion_2006}.

Our condition fails to be necessary because the first-order moments are not unique when considered optimization problem possesses multiple global optima, potentially leaving a strictly positive optimality gap. For the case of the unique global optimum, we have shown that the hierarchy eventually attains a zero optimality gap as the relaxation number approaches infinity. We note here that the possibility of the global minima multiplicity can almost be avoided in practice when the symmetry of structure and boundary conditions are exploited.

These theoretical results have been illustrated on five problems, which indicate the merits and weaknesses of the presented strategy. First, all of our test problems exhibited a~rapid convergence of the hierarchy, allowing for extraction of all global solutions based on the flat extension theorem of \citet{Curto_1996}. However, the computational complexity is currently fairly high, also when compared to investigated local optimization techniques, leaving the ability to compute proven global optima for small-scale optimization problems only. Yet, even for middle-scale problems, the hierarchy still provides a~sequence of upper bounds of reasonable qualities, and, especially, the certificate of their $\varepsilon$-optimality.

In the future, we plan to extend these results in multiple directions. First, we believe that Theorem \ref{th:zero} can be strengthened to certify a zero optimality gap when convergence of the hierarchy in a finite number of steps occurs. Second, the computational demands could be decreased by exploiting the structural sparsity via clique-based chordal decomposition in the spirit of~\citep{Kim_2010,Kocvara_2020}. Another research directions may explore eigenvalue constraints and optimization~\citep{Achtziger_2008,Tyburec_2019}, the minimum-weight setting, or, eventually, investigate performance of the hierarchy for different topology optimization formulations.

\begin{acknowledgement}
	We thank Edita Dvo{\v{r}}{\'{a}}kov{\'{a}} for providing us with her implementation of the \textsc{Mitc4} shell elements \citep{Dvorakova2015}.
	
	Marek Tyburec, Jan Zeman, and Martin Kru{\v{z}}{\'{i}}k acknowledge the support of the Czech Science Foundation project No. 19-26143X.
\end{acknowledgement}

\section*{Compliance with ethical standards}

{\small\noindent \textbf{Conflict of interest}\hspace{1mm} The authors declare that they have no conflict of interest.}

{\small\noindent \textbf{Replication of results}\hspace{1mm} Source codes are available at \citep{tyburec_marek_2020_4048828}.}

\appendix

\section{Relation to truss topology optimization}\label{app:tto}

The problem formulation \eqref{eq:nsdp} has already been known in the context of truss topology optimization \citep{Vandenberghe_1996}, for which the constraints \eqref{eq:pmi} reduce to linear matrix inequalities (LMI). Consequently, the feasible set is convex, allowing for an efficient solution of \eqref{eq:nsdp} by interior point methods, for example.

A natural question then arises: What happens when the rotational degrees of freedom are neglected, solving truss topology optimization problem instead of the frame one? To this goal, however, we must first satisfy the rather restrictive assumption that the truss ground structure is capable of carrying the loads $\mathbf{f}_j (\mathbf{a})$, i.e.,
\begin{equation}\label{eq:rel_image}
\mathbf{f}_j (\mathbf{a}) \in \mathrm{Im}\left(\mathbf{K}_{\mathrm{t},j}(\mathbf{a}) \right), \forall j \in \{1\dots n_\mathrm{lc}\},
\end{equation}
where $\mathbf{K}_{\mathrm{t},j} (\mathbf{a}) = \mathbf{K}_{j,0}+ \sum_{i=1}^{n_\mathrm{e}} \mathbf{K}_{j,i}^{(1)} a_i$.
Suppose now that $\mathbf{a}^{*}_\mathrm{t}$ are optimal cross-sections obtained from a solution to \eqref{eq:nsdp} with the terms $\mathbf{K}_{j,i}^{(2)}$ and $\mathbf{K}_{j,i}^{(3)}$ neglected, and $\bm{\omega}^\mathrm{T} \mathbf{c}^*_\mathrm{t}$ is the associated optimal objective function value, which can be computed from $\mathbf{a}_\mathrm{t}^*$ as
\begin{equation}\label{eq:trussC}
\bm{\omega}^\mathrm{T} \mathbf{c}^*_\mathrm{t} = \sum_{j=1}^{n_\mathrm{lc}}\left( \omega_j \left[\mathbf{f}_j(\mathbf{a}_\mathrm{t}^*)\right]^\mathrm{T} \left[\mathbf{K}_{\mathrm{t},j}(\mathbf{a}_\mathrm{t}^*) \right]^\dagger \mathbf{f}_j(\mathbf{a}_\mathrm{t}^*) \right).
\end{equation}
When the optimal cross-sections of a truss structure, $\mathbf{a}^{*}_\mathrm{t}$, are reused in a frame structure, the resulting objective function value changes to
\begin{equation}\label{eq:frameC}
\bm{\omega}^\mathrm{T} \mathbf{c}_\mathrm{f} = \sum_{j=1}^{n_\mathrm{lc}}\left( \omega_j \left[\mathbf{f}_j(\mathbf{a}_\mathrm{t}^*)\right]^\mathrm{T} \left[\mathbf{K}_{\mathrm{t},j}(\mathbf{a}_\mathrm{t}^*) + \mathbf{K}_{\mathrm{b},j}(\mathbf{a}_\mathrm{t}^*) \right]^\dagger \mathbf{f}_j(\mathbf{a}_\mathrm{t}^*) \right)
\end{equation}
with $\mathbf{K}_{\mathrm{b},j}(\mathbf{a}) = \sum_{i=1}^{n_\mathrm{e}}\left( c_\mathrm{II} \mathbf{K}_{j,i}^{(2)} a_{\mathrm{t},i}^2 + c_\mathrm{III}\mathbf{K}_{j,i}^{(3)} a_{\mathrm{t},i}^3\right)$.

To state a relation between $\bm{\omega}^\mathrm{T} \mathbf{c}^*_\mathrm{t}$ and $\bm{\omega}^\mathrm{T} \mathbf{c}_\mathrm{f}$ we recall a~useful lemma:

\begin{lemma}\label{lemma:sum} \citep{kovanic1979pseudoinverse}
	Let $\mathbf{A} \in \mathbb{S}^n$ and $\mathbf{B} \in \mathbb{R}^{n\times q}$. Then,
	\begin{equation}
	\begin{multlined}
	\left( \mathbf{A} + \mathbf{B}\mathbf{B}^\mathrm{T} \right)^\dagger = \mathbf{A}^\dagger \\- \mathbf{A}^\dagger \mathbf{B} \left(\mathbf{I} + \mathbf{B}^\mathrm{T}\mathbf{A}^\dagger \mathbf{B}\right)^{-1} \mathbf{B}^\mathrm{T} \mathbf{A}^\dagger + \left(\mathbf{B}^\dagger_\perp\right)^\mathrm{T} \mathbf{B}^\dagger_\perp
	\end{multlined}
	\end{equation}
	with $\mathbf{B}_\perp = \left(\mathbf{I} - \mathbf{A}\mathbf{A}^\dagger \right) \mathbf{B}$.
\end{lemma}

Using this lemma, we prove that $\bm{\omega}^\mathrm{T} \mathbf{c}_\mathrm{t}$ provides an upper bound for $\bm{\omega}^\mathrm{T} \mathbf{c}_\mathrm{f}$.
\begin{lemma}
	Suppose that $\bm{\omega}^\mathrm{T} \mathbf{c}^*$ is the optimal objective function value of the frame structure design problem \eqref{eq:nsdp} and \eqref{eq:rel_image} holds. Then, $\bm{\omega}^\mathrm{T} \mathbf{c}^* \le \bm{\omega}^\mathrm{T} \mathbf{c}_\mathrm{f} \le \bm{\omega}^\mathrm{T} \mathbf{c}^*_\mathrm{t}$.
\end{lemma}
\begin{proof}
	Because of $\mathbf{f}_j(\mathbf{a}_\mathrm{t}^*) \in \mathrm{Im}\left(\mathbf{K}_{\mathrm{t},j}(\mathbf{a}_\mathrm{t}^*)\right)$, we clearly have $\mathbf{f}_j (\mathbf{a}_\mathrm{t}^*) \in \mathrm{Im}\left( \mathbf{K}_{\mathrm{t},j}(\mathbf{a}_\mathrm{t}^*) + \mathbf{K}_{\mathrm{b},j}(\mathbf{a}_\mathrm{t}^*) \right)$. Therefore, $\mathbf{a}_\mathrm{t}^*$ is a feasible solution to the frame structure design problem \eqref{eq:nsdp} and the associated objective function is bounded from below by the global optimum $\bm{\omega}^\mathrm{T} \mathbf{c}^*$. Hence, $\bm{\omega}^\mathrm{T} \mathbf{c}^* \le \bm{\omega}^\mathrm{T} \mathbf{c}_\mathrm{f}$.
	
	For the other inequality, we express \eqref{eq:frameC} using Lemma~\ref{lemma:sum}. To this goal, let $\mathbf{K}_{\mathrm{b},j}(\mathbf{a}^*_\mathrm{t}) = \mathbf{B}_j \mathbf{B}_j^\mathrm{T}$, where $\mathbf{B}_j$ is a real matrix because $\mathbf{K}_{\mathrm{b},j}(\mathbf{a}^*_\mathrm{t}) \succeq 0$ by definition. Then, \eqref{eq:frameC} can be written as
	\begin{equation}\label{eq:kovanic}
	\bm{\omega}^\mathrm{T} \mathbf{c}_\mathrm{f} = \bm{\omega}^\mathrm{T} \mathbf{c}_\mathrm{t}^* - \bm{\omega}^\mathrm{T} \mathbf{c}_\mathrm{a} + \bm{\omega}^\mathrm{T} \mathbf{c}_\mathrm{b},
	\end{equation}
	where
	\begin{subequations}
		\begin{align}
		c_{\mathrm{a},j} &= \left[\mathbf{f}_j(\mathbf{a}_{\mathrm{t}}^*)\right]^\mathrm{T} \mathbf{A}_j^\dagger \mathbf{B}_j \left(\mathbf{I} + \mathbf{B}_j^\mathrm{T}\mathbf{A}_j^\dagger \mathbf{B}_j\right)^{-1} \mathbf{B}_j^\mathrm{T} \mathbf{A}_j^\dagger \mathbf{f}_j(\mathbf{a}_{\mathrm{t}}^*),\hspace{-1mm}\label{eq:kovanicA}\\
		c_{\mathrm{b},j} &= \left[\mathbf{f}_j(\mathbf{a}_{\mathrm{t}}^*)\right]^\mathrm{T} \left(\mathbf{B}^\dagger_{\perp,j}\right)^\mathrm{T} \mathbf{B}^\dagger_{\perp,j} \mathbf{f}_j(\mathbf{a}_{\mathrm{t}}^*),\label{eq:kovanicB}
		\end{align}
	\end{subequations}
	with $\mathbf{A}_j = \mathbf{K}_{\mathrm{t},j}(\mathbf{a}_\mathrm{t}^*)$. Clearly, Eq.~\eqref{eq:kovanicA} is non-negative. For \eqref{eq:kovanicB}, $\left(\mathbf{B}^\dagger_\perp\right)^\mathrm{T} \mathbf{B}^\dagger_\perp \in \mathrm{Ker}(\mathbf{A}_j)$, so that $\bm{\omega}^\mathrm{T} \mathbf{c}_\mathrm{b}$ vanishes. Hence, $\bm{\omega}^\mathrm{T} \mathbf{c}_\mathrm{f} = \bm{\omega}^\mathrm{T} \mathbf{c}_\mathrm{t}^* - \bm{\omega}^\mathrm{T} \mathbf{c}_\mathrm{a} \le \bm{\omega}^\mathrm{T} \mathbf{c}_\mathrm{t}^*$.
\end{proof}
Thus, when \eqref{eq:rel_image} holds true, the truss topology optimization produces an upper bound to the optimal objective of the frame structure topology optimization problem.

\bibliography{liter_abbr.bib}
\bibliographystyle{springernat}

\end{document}